\documentclass{article}
\usepackage[utf8]{inputenc}
\usepackage{graphicx}
\usepackage{amsfonts}
\usepackage{amsmath}
\usepackage{amssymb}
\usepackage{fancyhdr}
\usepackage{titlesec}
\usepackage{indentfirst}
\usepackage{booktabs}
\usepackage{verbatim}
\usepackage{color}
\usepackage{amsthm}
\usepackage{subfigure}
\usepackage{mathabx}

\usepackage[page,header]{appendix}
\usepackage{titletoc}

\newcommand{\rd}{\,\mathrm{d}}
\numberwithin{equation}{section}
\newtheorem{theorem}{Theorem}[section]
\newtheorem{lemma}[theorem]{Lemma}
\newtheorem{corollary}[theorem]{Corollary}

\newtheorem{proposition}[theorem]{Proposition}
\newtheorem{definition}[theorem]{Definition}
\newtheorem{remark}[theorem]{Remark}

\usepackage{mathtools}

\def\bx{{\bf x}}
\def\by{{\bf y}}
\def\bz{{\bf z}}

\def\cM{\mathcal{M}}

\def\cF{\mathcal{F}}

\def\cD{\mathcal{D}}
\def\cS{\mathcal{S}}
\def\cH{\mathcal{H}}

\def\supp{\textnormal{supp\,}}
\def\diam{\textnormal{diam\,}}
\def\dist{\textnormal{dist\,}}

\def\essinf{\textnormal{essinf\,}}

\begin{document}
	
\title{Existence of minimizers for interaction energies with external potentials}

\author{ Ruiwen Shu\footnote{Department of Mathematics, University of Georgia, Athens, GA 30602 (ruiwen.shu@uga.edu).}}

\maketitle

\begin{abstract}
	In this paper we study the existence of minimizers for interaction energies with the presence of external potentials. We consider a class of subharmonic interaction potentials, which include the Riesz potentials $|\bx|^{-s},\,\max\{0,d-2\}<s<d$ and its anisotropic counterparts. The underlying space is taken as $\mathbb{R}^d$ or a half-space with possibly curved boundary. We give a sufficient and almost necessary condition for the existence of minimizers, as well as the uniqueness of minimizers. The proof is based on the observation that the Euler-Lagrange condition for the energy minimizer is almost the same as that for the maximizer of the height functional, defined as the essential infimum of the generated potential. We also give two complimentary results: a simple sufficient condition for the existence of minimizers for general interaction/external potentials, and a slight improvement to the known result on the existence of minimizers without external potentials.
\end{abstract}

\section{Introduction}\label{sec_intro}

In this paper we study the minimizers of the interaction energy 
\begin{equation}\label{E}
	E_{W,U}[\rho] = \frac{1}{2}\int_{\mathbb{R}^d}\int_{\mathbb{R}^d}W(\bx-\by)\rd{\rho(\by)}\rd{\rho(\bx)} + \int_{\mathbb{R}^d} U(\bx)\rd{\rho(\bx)}\,,
\end{equation}
where $\rho\in \cM(\mathbb{R}^d)$ is a probability measure. $W:\mathbb{R}^d\rightarrow \mathbb{R}\cup\{\infty\}$ is an interaction potential, which satisfies the basic assumption
\begin{equation}\begin{split}
		\text{{\bf (W0)}: $W$} & \text{ is locally integrable, bounded from below} \\ & \text{and lower-semicontinuous, with $W(\bx)=W(-\bx)$.}
\end{split}\end{equation}
$U:\mathbb{R}^d\rightarrow \mathbb{R}\cup\{\infty\}$ is an external potential, which satisfies
\begin{equation}\begin{split}
		\text{{\bf (U0)}: $U$} & \text{ is locally integrable, bounded from below}  \text{ and lower-semicontinuous.}
\end{split}\end{equation}
As a consequence, $E_{W,U}[\rho]$ is always well-defined for any $\rho\in \cM(\mathbb{R}^d)$, taking values in $\mathbb{R}\cup \{\infty\}$. 

We are interested in the existence of minimizers of $E_{W,U}$ in $\cM(D)$, where the underlying domain is either the whole space\footnote{We will simply say `the minimizer of $E_{W,U}$' in case $D=\mathbb{R}^d$.} $\mathbb{R}^d$, or more generally, a closed subset of $\mathbb{R}^d$. If $D$ is bounded, then the existence can be easily obtained by a lower-semicontinuity argument for a minimizing sequence, c.f. \cite[Lemma 2.2]{SST15}. However, in general one does not expect  the existence of minimizers to hold unconditionally when $D$ is unbounded, in particular when $D=\mathbb{R}^d$, because a minimizing sequence can `escape to infinity' and does not converge weakly to a probability measure.

Without the external potential, the pairwise interaction energy 
\begin{equation}\label{EW}
	E_W[\rho] = \frac{1}{2}\int_{\mathbb{R}^d}\int_{\mathbb{R}^d}W(\bx-\by)\rd{\rho(\by)}\rd{\rho(\bx)}\,, 
\end{equation}
has been studied extensively. The existence of minimizers of $E_W$ is well-understood.
\begin{theorem}[\cite{SST15,CCP15}, with slight improvement]\label{thm_exist2}
	Assume {\bf (W0)}, and
	\begin{equation}
		\lim_{|\bx|\rightarrow\infty}W(\bx)=:W_\infty \in \mathbb{R}\cup\{\infty\}\,.
	\end{equation}
	Then there exists a minimizer of $E_W$ if and only if 
	\begin{equation}\label{thm_exist2_0}
		\textnormal{There exists $\rho\in \cM(\mathbb{R}^d)$ such that $E_W[\rho]\le \frac{1}{2}W_\infty$.}
	\end{equation}
	
	Furthermore, if there exists $\rho\in \cM(\mathbb{R}^d)$ such that $E_W[\rho]< \frac{1}{2}W_\infty$, i.e., 
	\begin{equation}\label{thm_exist2_1}
		\inf_{\rho\in\cM(\mathbb{R}^d)}E_W[\rho] < \frac{1}{2}W_\infty\,,
	\end{equation}
	then any minimizer of $E_W$ is compactly supported, and there exists $R>0$ depending on $W$ such that any minimizer $\rho$ satisfies $\diam(\supp\rho) \le R$.
\end{theorem}
We leave the further discussions and the proof of the new part in Theorem \ref{thm_exist2} to Section \ref{sec_exist2}. Some related results can also be found in \cite{CFT15}.

The uniqueness of minimizers of $E_W$, up to translation, can be obtained under the condition of linear interpolation convexity (LIC) \cite{Lop19,CS21,ST,DLM1,DLM2,Fra,Shu_convex}. Based on the LIC property and the sufficiency of the Euler-Lagrange condition, explicit formulas for the minimizers of $E_W$ have been obtained for some cases of the power-law potential $W_{a,b}(\bx) = \frac{|\bx|^a}{a} - \frac{|\bx|^b}{b}$, which include minimizers as a locally integrable function \cite{CH17,CS21,Fra,Shu_expli}, spherical shells \cite{BCLR13_2,DLM1,FM} and discrete measures \cite{DLM2}. Minimizers for some anisotropic interaction energies were also derived from this approach in 2D and 3D \cite{MRS,CMMRSV1,CMMRSV2,MMRSV1,MMRSV2,CS_2D,CS_3D}, and very recently in higher dimensions \cite{FMMRSV}. For the energy $E_{W,U}$ involving the external potentials, similar convexity conditions and explicit minimizers, as well as some existence theory, were also obtained in some recent works \cite{ST,DOSW,DOSW25,CMSVW}.

In this paper we focus on the existence of minimizers of $E_{W,U}$. It is not too hard to find sufficient conditions for the existence of minimizers: we will present a known condition in Theorem \ref{thm_saff}, as well as a new one in Theorem \ref{thm_exist1} whose proof is short. However, it seems to be very challenging to find \emph{necessary} and sufficient conditions for the existence of minimizers, analogous to the condition \eqref{thm_exist2_0} for $E_W$. {One known such condition is stated in Theorem \ref{thm_zorii} for a special type of $U$, but no necessary and sufficient condition is known for general $U$ satisfying {\bf (U0)}.}

We will introduce a new class of subharmonic interaction potentials in Section \ref{sec_ECintro}, called \emph{essentially convex} potentials, which generalizes the singular power-law repulsion (also known as Riesz potentials) $W(\bx)=-\frac{|\bx|^b}{b}$ with $-d<b<\min\{2-d,0\}$ and its anisotropic counterparts. Our main result, Theorem \ref{thm_main}, gives a sufficient and `almost necessary' condition for the existence of minimizers when $W$ is essentially convex, and the underlying domain $D$ is either $\mathbb{R}^d$ or a half-space with possibly curved boundary.

This result relies on the key property that $\Delta W>0$ on $\mathbb{R}^d\backslash \{0\}$. Without this property, say, when $d\ge 3$ and $W(\bx) = -\frac{|\bx|^b}{b}$ with $2-d<b<0$, we will show in Theorem \ref{thm_counter} that the same condition in Theorem \ref{thm_main} can no longer guarantee the existence of minimizer. For such superharmonic potentials $W$, it is still widely open to find a necessary and sufficient condition for the existence of minimizer of $E_{W,U}$. 

\subsection{Some sufficient conditions for the existence of minimizers}

For the power-law repulsions $W(\bx)=|\bx|^{-s},\,0<s<d$, \cite{DOSW} gives a sufficient condition for the existence of minimizers.
\begin{theorem}\cite[Theorem 2.4(i)]{DOSW}\label{thm_saff}
	Assume $W(\bx)=|\bx|^{-s},\,0<s<d$, and $U$ satisfies {\bf (U0)}. Assume $\lim_{|\bx|\rightarrow\infty}U(\bx) =U_{\infty}\in \mathbb{R}$, and
	\begin{equation}\label{saffcond}
		|\bx|^s(U(\bx)-U_\infty) \le -1\,,
	\end{equation}
	for any sufficiently large $|\bx|$. Then there exists a minimizer of $E_{W,U}$.
\end{theorem}
We remark that \cite{DOSW} can handle more general situations, including a general unbounded underlying domain other than $\mathbb{R}^d$, or \eqref{saffcond} having certain failure set.

{
We also mention the significant works of N. Zorii \cite{Zorii_old1,Zorii_old2,Zorii_old3,Zorii23_1,Zorii23_2} on the existence of energy minimizers with external potentials. In particular, for $d\ge 2$, $W(\bx)=|\bx|^{-s},\,d-2<s<d$ and a special class of $U$, \cite{Zorii23_2} gives a necessary and sufficient condition on the existence. We formulate it here when the underlying space is $\mathbb{R}^d$.
\begin{theorem}[{\cite[Theorem 2.6]{Zorii23_2}}]\label{thm_zorii}
	Assume $d\ge 2$, $W(\bx)=|\bx|^{-s},\,d-2<s<d$, and $U=-W*\omega$ for some positive measure $\omega$ on $\mathbb{R}^d$. Then there exists a minimizer of $E_{W,U}$ if and only if $\omega(\mathbb{R}^d)\ge 1$.
\end{theorem}
Under the same assumptions on $W$ and $U$, \cite{Zorii23_2} also gives a necessary and sufficient condition for general underlying domains, using a concept called `inner balayage' of a measure $\omega$ onto the underlying domain.
}

Then, for general interaction potentials $W$, we state a simple condition in the spirit of Theorem \ref{thm_exist2}, showing that if $U$ is sufficiently large at infinity, then there exists a minimizer of $E_{W,U}$.
\begin{theorem}\label{thm_exist1}
	Assume {\bf (W0)}, {\bf (U0)} and $\lim_{|\bx|\rightarrow\infty}U(\bx) =U_{\infty}\in \mathbb{R}\cup\{\infty\}$. If
	\begin{equation}\label{thm_exist1_1}
		\inf_{\rho\in\cM(\mathbb{R}^d)}E_{W,U}[\rho] < \frac{1}{2}(U_\infty + \inf U + \inf W)\,,
	\end{equation}
	then there exists a minimizer of $E_{W,U}$, and any minimizer of $E_{W,U}$ is compactly supported.
\end{theorem}
This theorem will be proved in Section \ref{sec_exist1}.

The conditions in Theorems \ref{thm_saff} and \ref{thm_exist1} are far from being necessary. In fact, for Theorem \ref{thm_saff}, by considering $W(\bx)=|\bx|^{-s},\,0<s<d$ and $U = -C \chi_{\overline{B(0;1)}}$ where $C>0$ is large, it is easy to see that \eqref{thm_exist1_1} is true. So there exists a minimizer, but \eqref{saffcond} is clearly false. For Theorem \ref{thm_exist1}, if one considers $U=0$, then \eqref{thm_exist1_1} degenerates into $\inf_{\rho\in\cM(\mathbb{R}^d)}E_{W}[\rho] < \frac{1}{2} \inf W$ which is never true, and thus it cannot be a necessary condition. {Theorem \ref{thm_zorii}, although providing a necessary and sufficient condition, only works when $U$ can be expressed in the form $U=-W*\omega$ and does not work for more general $U$.}

\subsection{The Euler-Lagrange condition and strong convexity}

For the purpose of stating our main result, in this subsection we recall the Euler-Lagrange condition (also known as the Frostman's inequality) for minimizers, the convexity theory which guarantees the uniqueness of minimizers, and their relation. For any $\rho\in\cM(\mathbb{R}^d)$, define the \emph{generated potential}
\begin{equation}
	V[\rho] := W*\rho+U\,.
\end{equation}

\begin{lemma}[Euler-Lagrange condition]\label{lem_EL}
	Assume $W,U$ satisfy {\bf (W0)}{\bf (U0)}. 	Let $D\subset \mathbb{R}^d$ be a closed set such that $\inf_{\cM(D)}E_{W,U}<\infty$.  If $\rho\in\cM(D)$ is a minimizer of $E_{W,U}$ in $\cM(D)$, then
	\begin{equation}\label{lem_EL_1}
		V[\rho] \le C_0,\quad \textnormal{on } \supp\rho\,;\qquad V[\rho] \ge C_0,\quad\textnormal{a.e. }D\,,
	\end{equation}
	and
	\begin{equation}\label{lem_EL_2}
		V[\rho]=C_0,\quad \rho\textnormal{-a.e.}
	\end{equation}
	for some $C_0\in\mathbb{R}$.
\end{lemma}
This lemma can be proved by the same argument as \cite[Theorem I.1.3]{saffbook} for the logarithmic potential, and it is analogous to \cite[Theorem 4]{BCLR13_1} which is the Euler-Lagrange condition for $E_W$. A refined version is discussed in Remark \ref{rmk_EL} below.

Next we introduce the following conditions from \cite{saffbook2}, which were originally formulated more generally for integral kernels.
\begin{definition}[{\cite[Definition 4.2.5]{saffbook2}}]
	Assume $W$ satisfies {\bf (W0)}. We say $W$ satisfies the \emph{strictly positive definite} (SPD) property if for any nontrivial signed measure $\mu$ on $\mathbb{R}^d$ and $E_W[\mu]$ well-defined (as a number in $[-\infty,\infty]$), one has $E_W[\mu]>0$. We say $W$ satisfies the \emph{conditionally strictly positive definite} (CSPD) property if for any nontrivial signed measure $\mu$ on $\mathbb{R}^d$ with $\mu(\mathbb{R}^d)=0$ and $E_W[\mu]$ well-defined, one has $E_W[\mu]>0$.
\end{definition}
We give an equivalent formulation of the CSPD property in terms of energy convexity along linear interpolations.
\begin{lemma}\label{lem_CSPD}
	CSPD is equivalent to the following statement: 
	\begin{equation}\label{lem_CSPD_1}\begin{split}
		&\textnormal{For any distinct $\rho_0,\rho_1\in \cM(\mathbb{R}^d)$ with $E_W[\rho_i]<\infty,\,i=0,1$,} \\ & \textnormal{the function $t\mapsto E_W[(1-t)\rho_0+t\rho_1]$ is strictly convex on $t\in [0,1]$.}
	\end{split}\end{equation}
\end{lemma}
Notice that under the above condition,  the same convexity is also true for $E_{W,U}$ because the contribution from the $U$ part is linear in $t$.
\begin{proof}
	If CSPD holds, then for any distinct $\rho_0,\rho_1\in \cM(\mathbb{R}^d)$ with $E_W[\rho_i]<\infty,\,i=0,1$, one has $E_W[\rho_0+\rho_1]<\infty$ due to \cite[Lemma 4.2.6]{saffbook2}. This justifies the calculation
	\begin{equation}
		\frac{\rd^2}{\rd t^2}E_W[(1-t)\rho_0+t\rho_1] = 2E_W[\rho_1-\rho_0]\,,
	\end{equation}
	which is strictly positive due to CSPD. Therefore \eqref{lem_CSPD_1} holds. 
	
	On the other hand, if \eqref{lem_CSPD_1} holds, then for any nontrivial signed measure $\mu$ on $\mathbb{R}^d$ with $\mu(\mathbb{R}^d)=0$ and $E_W[\mu]$ well-defined, if $E_W[\mu]=\infty$ then the conclusion is obtained. Otherwise, we have $E_W[\mu]<\infty$, and we decompose $\mu=\mu_+-\mu_-$ as its positive and negative parts.  Then $E_W[\mu_+],E_W[\mu_-]$ are also finite.  Since $\mu(\mathbb{R}^d)=0$, we have $\mu_+(\mathbb{R}^d)=\mu_-(\mathbb{R}^d)>0$. By rescaling, we may assume that $\mu_+,\mu_-$ are probability measures without loss of generality. Then, due to \eqref{lem_CSPD_1}, we have that $t\mapsto E_W[(1-t)\mu_-+t\mu_+]$ is strictly convex on $t\in [0,1]$. This justifies that $E_W[\mu_-+\mu_+]<\infty$, and thus
	\begin{equation}
		\frac{\rd^2}{\rd t^2}E_W[(1-t)\mu_-+t\mu_+] = 2E_W[\mu_+-\mu_-] = 2E_W[\mu]\,,
	\end{equation}
	is a positive number. This gives CSPD.
\end{proof}
\begin{remark}
	Clearly SPD is stronger than CSPD. Also notice that CSPD, in view of Lemma \ref{lem_CSPD},  is stronger than the LIC condition introduced in \cite{CS21}, which only requires the convexity along linear interpolations when $\rho_0$ and $\rho_1$ are compactly supported and have the same center of mass.
\end{remark}

An application of Lemma \ref{lem_CSPD} gives the following result directly.
\begin{lemma}\label{lem_SCuni}
	Assume $W,U$ satisfy {\bf (W0)}{\bf (U0)} and $W$ is CSPD. Let $D\subset \mathbb{R}^d$ be a closed set such that $\inf_{\cM(D)}E_{W,U}<\infty$. Then the minimizer of $E_{W,U}$ in $\cM(D)$, if exists, has to be unique.
\end{lemma}

\color{black}

For CSPD potentials, the Euler-Lagrange condition is also sufficient for minimizers. Similar result in the absence of external potentials was derived in \cite[Theorem 2.4]{CS21}. We recall a technical condition on $W$ from \cite{CS21}:
\begin{equation}\begin{split}
		\textnormal{{\bf (W1)}:} & \textnormal{ $W$ is continuous on $\mathbb{R}^d\backslash \{0\}$, and $W(0)=\lim_{\bx\rightarrow0}W(\bx)\in \mathbb{R}\cup\{\infty\}$; }\\
		& \textnormal{for any $R>0$, }\frac{1}{|B(\bx;\epsilon)|}\int_{B(\bx;\epsilon)}W(\by)\rd{\by} \le C_1(R) + C_2(R) W(\bx),\quad \forall |\bx|<R,\, 0<\epsilon<1\,.
\end{split}\end{equation}

\begin{lemma}\label{lem_suf}
	Under the same assumptions as Lemma \ref{lem_SCuni}, if $\rho\in\cM(D)$ satisfies
	\begin{equation}\label{lem_suf_1}
		V[\rho] \le C_0,\quad \textnormal{on } \supp\rho\,;\qquad V[\rho] \ge C_0,\quad\textnormal{on }D\,,
	\end{equation}
	for some $C_0\in\mathbb{R}$, then $\rho$ is the unique minimizer of $E_{W,U}$ in $\cM(D)$. 
	
	If $D=\mathbb{R}^d$, $W$ satisfies {\bf (W1)} and $U$ is continuous, then the same is true under a slightly weaker condition \eqref{lem_EL_1}.
\end{lemma}

\begin{proof}
Let $\rho_1\in\cM(D)$ with $\rho_1\ne \rho$ and $E_{W,U}[\rho_1]<\infty$. Define
\begin{equation}
	\rho_t = (1-t)\rho + t\rho_1,\quad \forall t\in [0,1]\,.
\end{equation}
Then $E_W[\rho_t]$ is strictly convex by Lemma \ref{lem_CSPD}. This implies that $E_{W,U}[\rho_t]$ is strictly convex because the contribution from $U$ is linear, which in particular implies that $E_{W,U}[\rho+\rho_1]<\infty$. Then we may calculate
\begin{equation}
	\frac{\rd}{\rd t}\Big|_{t=0} E_{W,U}[\rho_t] = \int_D V[\rho]\rd{\rho_1} - \int_D V[\rho]\rd{\rho} \ge C_0 - C_0 = 0\,,
\end{equation}
by \eqref{lem_suf_1}. This, with the strict convexity of $E_{W,U}[\rho_t]$, shows that $E_{W,U}[\rho_1]>E_{W,U}[\rho]$. Therefore we see that $\rho$ is the unique minimizer of $E_{W,U}$ in $\cM(D)$.

If $D=\mathbb{R}^d$, $W$ satisfies {\bf (W1)}, $U$ is continuous and the condition \eqref{lem_EL_1} is assumed, then we prove by contradiction. Assume $\rho_1\in\cM(\mathbb{R}^d)$ satisfies $E_{W,U}[\rho_1] < E_{W,U}[\rho]$. By a truncation argument (c.f. Lemma \ref{lem_trunres}), we may assume $\rho_1$ is compactly supported without loss of generality. Then, using a mollifier $\phi_\epsilon$ (c.f. the paragraph before Section \ref{sec_EC}), with the aid of \cite[Lemma 2.5]{CS21} and the continuity of $U$, we have
\begin{equation}
	\lim_{\epsilon\rightarrow0} E_{W,U}[\rho_1*\phi_\epsilon] = E_{W,U}[\rho_1]\,.
\end{equation}
Therefore, by taking $\epsilon$ sufficiently small, we get a smooth and compactly supported $\tilde{\rho}_1:=\rho_1*\phi_\epsilon\in\cM(\mathbb{R}^d)$ such that $E_{W,U}[\tilde{\rho}_1] < E_{W,U}[\rho]$. 

Notice that $\int_{\mathbb{R}^d} V[\rho]\rd{\tilde{\rho}_1} \ge C_0$ under the weaker condition \eqref{lem_EL_1} since $\tilde{\rho}_1$ is continuous. Then the same calculation as above with $\rho_1$ replaced by $\tilde{\rho}_1$ shows that $E_{W,U}[\tilde{\rho}_1]>E_{W,U}[\rho]$, leading to a contradiction. Therefore we see that $\rho$ is the unique minimizer of $E_{W,U}$ in $\cM(\mathbb{R}^d)$. 

\end{proof}

\begin{remark}\label{rmk_EL}
	For the power-law repulsions $W(\bx)=-\frac{|\bx|^b}{b},\,-d<b<0$, \cite[Theorem 2.1]{DOSW} shows that the condition 
	\begin{equation}
		V[\rho] \le C_0,\quad \textnormal{on } \supp\rho\,;\qquad V[\rho] \ge C_	0,\quad\textnormal{q.e. }D\,,
	\end{equation}
	is a necessary and sufficient condition for minimizer. Here q.e. means that any probability measure supported on the failure set has infinite $E_W$ energy. This condition is weaker than \eqref{lem_suf_1} but stronger than \eqref{lem_EL_1}\eqref{lem_EL_2}.
\end{remark}

\subsection{Essentially convex potentials and the main results}\label{sec_ECintro}

We define the following class of potentials, whose prototype is the singular power-law repulsion $-\frac{|\bx|^b}{b}$ with $-d<b<\min\{2-d,0\}$.
\begin{definition}
	Let $W\in C^2(\mathbb{R}^d\backslash \{0\})$ satisfy {\bf (W0)}{\bf (W1)}. We say $W$ is \emph{essentially convex} if 	\begin{equation}
		\Delta W>0,\quad \textnormal{on }\mathbb{R}^d\backslash \{0\}\,,
	\end{equation}
	and there holds the estimates
	\begin{equation}\label{EC1}
		|W| \lesssim |\bx|^{b},\,0< |\bx| <1;\quad |W| \lesssim |\bx|^{a},\,|\bx|\ge 1\,,
	\end{equation}
	and
	\begin{equation}\label{EC2}
		|\nabla W| \lesssim |\bx|^{b-1},\,0< |\bx| <1;\quad |\nabla W| \lesssim |\bx|^{a-1},\,|\bx|\ge 1\,,
	\end{equation}
	and
	\begin{equation}\label{EC3}
		|\Delta W| \lesssim |\bx|^{b-2},\,0< |\bx| <1;\quad |\Delta W| \lesssim |\bx|^{a-2},\,|\bx|\ge 1\,,
	\end{equation}
	for some $-d<b<\min\{2-d,0\}$ and $-d<a<\min\{2-d,0\}$. 
\end{definition}

\begin{remark}(Other examples of essentially convex potentials)
	Consider the anisotropic repulsive potential
	\begin{equation}
		W(\bx) = -\frac{|\bx|^b}{b} \Big(1+\alpha \omega(\frac{\bx}{|\bx|})\Big)\,,
	\end{equation}
	with $-d<b<\min\{2-d,0\}$, where $\alpha>0$ and $\omega:S^{d-1}\rightarrow \mathbb{R}_{\ge 0}$ is a smooth function. Then one can calculate
	\begin{equation}
		\Delta W(\bx) = \Big(-(b+d-2) - \frac{\alpha}{b}\Delta_{S^{d-1}}\omega\Big)|\bx|^{b-2}\,,
	\end{equation}
	where $-(b+d-2)>0$. Therefore it is essentially convex whenever $0\le \alpha < \frac{-(b+d-2)b}{\min  \Delta_{S^{d-1}}\omega}$ (here, notice that $\min  \Delta_{S^{d-1}}\omega<0$ for any nontrivial $\omega$). As will be seen in Lemma \ref{lem_convexpL}, such potentials have strictly positive Fourier transform, and thus the theory for the energy minimizers (in the presence of quadratic confinement) applies, see the latest results in \cite{FMMRSV}. On the other hand, it is not hard to see that an anisotropic potential with positive Fourier transform is not necessarily essentially convex. 
	
	There are also essentially convex potentials such that its singularity at 0 and its tail at infinity have different powers. For example, consider
	\begin{equation}
		W(\bx) = \sum_{j=1}^J A_j\cdot \Big(-\frac{|\bx|^{b_j}}{b_j}\Big) \,,
	\end{equation}
	where $A_j>0$ and $-d<b_1< \dots < b_J<\min\{2-d,0\}$. Then $W$ is essentially convex, $W(\bx)\sim -A_1\frac{|\bx|^{b_1}}{b_1}$ near 0, and $W(\bx)\sim -A_J\frac{|\bx|^{b_J}}{b_J}$ at infinity.
\end{remark}

We first establish the convexity property of essentially convex potentials.
\begin{lemma}\label{lem_EC}
	Essentially convex potentials are SPD.
\end{lemma}

Then we state our main result, which is a sufficient and almost necessary condition for the existence of minimizers of $E_{W,U}$ when $W$ is essentially convex. At the same time, the uniqueness of minimizer is also obtained as a direct consequence of the SPD property of $W$.
\begin{theorem}\label{thm_main}
	Assume {\bf (W0)}{\bf (W1)}{\bf (U0)}, $W$ is essentially convex, $W\ge 0$ on $\mathbb{R}^d$, $U$ is continuous, and $\lim_{|\bx|\rightarrow\infty}U(\bx)=U_\infty\in\mathbb{R}$. Let the underlying domain $D$ be one of the following: 
	\begin{itemize}
		\item $D=\mathbb{R}^d$;
		\item $D=[x_0,\infty)$ for some $x_0\in\mathbb{R}$, if $d=1$;
		\item $D=\{\bx=(x_1,\dots,x_d)\in\mathbb{R}^d: x_d \ge \Phi(\hat{\bx})\}$, where $\hat{\bx}:=(x_1,\dots,x_{d-1})$ and $\Phi:\mathbb{R}^{d-1}\rightarrow \mathbb{R}$ is a continuous function, if $d=2$.
	\end{itemize}
	Assume there exists $\rho_\sharp\in\cM(D)$ that is a compactly supported locally integrable function, $\supp\rho_\sharp=\bar{\cS}$ for some open set $\cS$ with $\partial \cS$ having zero Lebesgue measure, $W*\rho_\sharp$ is continuous on $\mathbb{R}^d$, together with
	\begin{equation}\label{thm_main_1}
		\sup_{\cS}V[\rho_\sharp] < U_\infty\,.
	\end{equation}
	Then there exists a unique minimizer of $E_{W,U}$ in $\cM(D)$, and this minimizer is compactly supported.
\end{theorem}

\begin{remark}
	A half-space-type domain constraint as stated in Theorem \ref{thm_main} may have a significant influence on the qualitative behavior of the minimizers of $E_{W,U}$. For example, it was observed in \cite{BT} that some mass of the minimizer may concentrate on the boundary of the domain.
\end{remark}

The conditions in Theorem \ref{thm_main} for $\rho_\sharp$ are almost necessary for the existence of minimizers, in the sense that any minimizer of $E_{W,U}$ almost satisfies the assumptions for $\rho_\sharp$. To be precise, we have the following.
\begin{proposition}\label{prop_nece}
	Under the same assumptions on $W$, $U$ and $D$ as in Theorem \ref{thm_main}, suppose there exists a minimizer $\rho_\infty$  of $E_{W,U}$ in $\cM(D)$, whose $C_0$ value as stated in Lemma \ref{lem_EL} satisfies 
	\begin{equation}\label{prop_nece_1}
		C_0 < U_\infty\,.
	\end{equation}
	Assume $W*\rho_\infty$ is continuous on $\mathbb{R}^d$. Then $\rho_\infty$ is compactly supported, and some translation and mollification of $\rho_\infty$ satisfies the assumptions for $\rho_\sharp$ in Theorem \ref{thm_main}.
\end{proposition}

\begin{remark}
	Notice that under the same assumptions on $W$ and $U$ as in Theorem \ref{thm_main}, any minimizer $\rho_\infty$ of $E_{W,U}$ necessarily satisfies $C_0 \le U_\infty$ because $\liminf_{|\bx|\rightarrow\infty}V[\rho_\infty](\bx) = U_\infty$. Therefore the condition \eqref{prop_nece_1} is almost a necessary condition for minimizers, up to the equal sign.
	
	We also notice that if $W(\bx)=-\frac{|\bx|^b}{b},\,-d<b<\min\{2-d,0\}$ and $U$ is $C^2$, then an easy adaption of \cite[Proposition 3.8]{CDM} shows that any compactly supported minimizer $\rho_\infty$ of $E_{W,U}$ in $\cM(\mathbb{R}^d)$ satisfies the condition that $W*\rho_\infty$ is continuous.
\end{remark}

{
\begin{remark}
	The `if' part of Theorem \ref{thm_zorii} can be derived from Theorem \ref{thm_main} up to regularity, tail and the equal sign issues. In fact, the $W$ stated in Theorem \ref{thm_zorii} satisfies the assumptions of Theorem \ref{thm_main}. Suppose $U=-W*\omega$ for some positive measure $\omega$ with $\omega(\mathbb{R}^d)>1$. Also assume that $\omega$ is a compactly supported continuous function, whose support is $\bar{\cS}$ for some open set $\cS$ with $\partial \cS$ having zero Lebesgue measure. Then, denoting $\rho_\sharp = \frac{\omega}{\omega(\mathbb{R}^d)}$, we have
	\begin{equation}
		V[\rho_\sharp] = \Big(\frac{1}{\omega(\mathbb{R}^d)}-1\Big)(W*\omega)
	\end{equation}
	which satisfies \eqref{thm_main_1} since $W*\omega$ is continuous and strictly positive, $\cS$ is bounded, and $\omega(\mathbb{R}^d)>1$. Therefore the existence of minimizer follows from Theorem \ref{thm_main}.
\end{remark}
}


%

For $W$ which is not essentially convex, the condition \eqref{thm_main_1} cannot imply the existence of minimizer in general, as stated in the following theorem.
\begin{theorem}\label{thm_counter}
	Assume $d\ge 3$ and $W(\bx) = -\frac{|\bx|^b}{b}$ with $2-d<b<0$. Then there exists a smooth radial $U$ satisfying $\lim_{|\bx|\rightarrow\infty}U(\bx)=0$, such that there exists a smooth and radial $\rho_\sharp\in\cM(\mathbb{R}^d)$ supported on a ball $\overline{B(0;R)}$ with $\sup_{B(0;R)} V[\rho_\sharp] < 0$, but $E_{W,U}$ does not admit a minimizer.
\end{theorem}

\subsection{Sketch of proof}

The proof of Lemma \ref{lem_EC} is based a representation formula for essentially convex potentials (Lemma \ref{lem_convexpL}), which allows us to express the Fourier transform of $W$ explicitly in terms of $\Delta W$. This allows us to show the positivity of $\hat{W}$ for essentially convex potentials, and derive its SPD property.

The key idea of the proof of Theorem \ref{thm_main}, in the case $D=\mathbb{R}^d$, is the \emph{height} functional
\begin{equation}\label{H0}
	\cH_{\mathbb{R}^d}[\rho] = \essinf_{\bx\in \mathbb{R}^d} V[\rho](\bx)\,,
\end{equation}
(see a more general version in \eqref{H}). $\cH_{\mathbb{R}^d}$ is an \emph{upper}-semicontinuous functional, and we observe that the Euler-Lagrange condition for the \emph{maximizer} of $\cH_{\mathbb{R}^d}[\rho]$ is precisely \eqref{lem_EL_1}, see Lemma \ref{lem_Hcont} and Theorem \ref{thm_ELH} for more general versions involving domain constraints. Therefore, it suffices to show the existence of maximizer of $\cH_{\mathbb{R}^d}$, which has to be the minimizer of $E_{W,U}$ due to Lemma \ref{lem_suf}. The existence of maximizer of $\cH_{\mathbb{R}^d}$ can be guaranteed via a threshold condition \eqref{thm_Hexist} in Theorem \ref{thm_Hexist}. It turns out that the condition \eqref{thm_main_1} can imply the condition \eqref{thm_Hexist} via a comparison argument in Lemma \ref{lem_rhosharp}.  

The condition that $\Delta W>0\, \textnormal{ on }\mathbb{R}^d\backslash \{0\}$ is crucially used in proof of Theorem \ref{thm_ELH}. It guarantees that an operation called `microscopic diffusion', outlined in Lemma \ref{lem_posDelta2}, always increases the generated potential far away. This idea was first introduced in the author's work with Wang \cite[Lemma 2.4]{SW21} in a one-dimensional setting.

To treat other underlying domains, the only extra difficulty is the slight mismatch of the Euler-Lagrange conditions: In Theorem \ref{thm_ELH} (for the maximizer of the height functional) one always has a possible failure set of Lebesgue measure zero, but this is not allow if one wants a sufficient condition for the energy minimizer on a general underlying domain, c.f. \eqref{lem_suf_1}. Therefore, one has to propose the correct assumption on $D$ so that the weaker condition \eqref{lem_EL_1} can imply the energy minimizer. Lemma \ref{lem_suf} already handled the case $D=\mathbb{R}^d$ via a known mollification argument. For curved half-spaces as stated in Theorem \ref{thm_main}, this is done in Lemma \ref{lem_sufhalf} by a mollification and translation argument based on the modulus of continuity of $U$ and $\Phi$.

To prove Theorem \ref{thm_counter}, we first give a sufficient condition for the non-existence of minimizer in Lemma \ref{lem_nonex}. Then, in Corollary \ref{cor_counter} we construct $U$ and $\rho_\sharp$ so that Lemma \ref{lem_nonex} applies and the inequality $\sup_{B(0;R)} V[\rho_\sharp] < 0$ can be verified explicitly.

\subsection{Organization of the paper}

The rest of the paper is organized as follows: in Section \ref{sec_EC} we give a representation formula, Lemma \ref{lem_convexpL}, for essentially convex potentials, and derive Lemma \ref{lem_EC}. In Section \ref{sec_H} we introduce the height functional and prove its basic properties, including the upper-semicontinuity, existence of maximizer and Euler-Lagrange condition. In Section \ref{sec_main} we prove the main results, Theorem \ref{thm_main}, Proposition \ref{prop_nece} and Theorem \ref{thm_counter}. In Section \ref{sec_compli} we prove the complimentary results, Theorem \ref{thm_exist1} and the new part of Theorem \ref{thm_exist2}.

Within this paper, a \emph{mollifier} $\phi$ refers to a smooth radially-decreasing nonnegative function supported on $\overline{B(0;1)}$ with $\int_{\mathbb{R}^d}\phi\rd{\bx}=1$, with the scaling notation $\phi_\epsilon(\bx) = \frac{1}{\epsilon^d}\phi(\frac{\bx}{\epsilon})$.

\section{Convexity of essentially convex potentials}\label{sec_EC}

We give an explicit representation formula for functions $W$ satisfying the regularity assumptions in the definition of essential convexity. Then one can see that an essentially convex potential $W$ satisfies $\hat{W}>0$. Based on this, we prove Lemma \ref{lem_EC} by a standard argument at the end of this section.

For the Fourier transform, we adopt the tradition that $\hat{f}(\xi)=\int_{\mathbb{R}^d}f(\bx)e^{-2\pi i \bx\cdot\xi}\rd{\bx}$ for $f\in L^1(\mathbb{R}^d)$.

Denote
\begin{equation}
	N(\bx) = \left\{\begin{split}
		-\frac{1}{(2-d)|S^{d-1}|}|\bx|^{2-d}, & \quad d\ne 2 \\
		-\frac{1}{2\pi}\ln|\bx|, & \quad d= 2 \\
	\end{split}\right.\,,
\end{equation}
as the repulsive Newtonian potential, satisfying $\Delta N = -\delta$. 

\begin{lemma}\label{lem_convexpL}
	Let $W\in C^2(\mathbb{R}^d\backslash \{0\})$ with {\bf (W0)}. Assume $W$ satisfies \eqref{EC1}\eqref{EC2}\eqref{EC3} for some $-d<b<\min\{2-d,0\}$ and $-d<a<\min\{2-d,0\}$. Then
	\begin{equation}\label{lem_convexpL_0}
		W(\bx) = \lim_{\epsilon\rightarrow0,\,R\rightarrow\infty} \int_{\epsilon \le |\by|\le R} \Big(N(\bx)-\frac{1}{2}N(\bx-\by)-\frac{1}{2}N(\bx+\by)\Big)\Delta W(\by)\rd{\by}\,,
	\end{equation}
	where the limit holds uniformly on compact sets in $\mathbb{R}^d\backslash \{0\}$, as well as in the sense of tempered distribution. It follows that
	\begin{equation}\label{lem_convexpL_1}
		\hat{W}(\xi) = \lim_{\epsilon\rightarrow0,\,R\rightarrow\infty} \frac{1}{4\pi^2}\int_{\epsilon \le |\by|\le R} |\xi|^{-2}\big(1-\cos(2\pi\by\cdot\xi)\big)\Delta W(\by)\rd{\by}\,,
	\end{equation}
	in the sense of tempered distribution, and $\hat{W}$ is a locally integrable function.
\end{lemma}

\begin{proof}
Denote
\begin{equation}
	\Omega_{\epsilon,R} = \{\by:\epsilon < |\by|< R\},\quad W_{\epsilon,R}(\bx) =  \int_{\Omega_{\epsilon,R}} \Big(N(\bx)-\frac{1}{2}N(\bx-\by)-\frac{1}{2}N(\bx+\by)\Big)\Delta W(\by)\rd{\by}\,.
\end{equation}
Then each $W_{\epsilon,R}$ is continuous in $\Omega_{\epsilon,R}$. Then, for any $\bx\in \Omega_{\epsilon,R}$, we use integration by parts to calculate
\begin{equation}\begin{split}
		W_{\epsilon,R}(\bx) = & \int_{\partial\Omega_{\epsilon,R}} \Big(N(\bx)-\frac{1}{2}N(\bx-\by)-\frac{1}{2}N(\bx+\by)\Big)\nabla W(\by)\cdot \vec{n}\rd{S(\by)} \\
		& -  \frac{1}{2}\int_{\Omega_{\epsilon,R}} \Big(\nabla N(\bx-\by)-\nabla N(\bx+\by)\Big)\cdot\nabla W(\by)\rd{\by} \\
		= & \int_{\partial\Omega_{\epsilon,R}} \Big(N(\bx)-\frac{1}{2}N(\bx-\by)-\frac{1}{2}N(\bx+\by)\Big)\nabla W(\by)\cdot \vec{n}\rd{S(\by)} \\
		&-  \frac{1}{2}\int_{\partial\Omega_{\epsilon,R}} \Big(\nabla N(\bx-\by)-\nabla N(\bx+\by)\Big)W(\by)\cdot \vec{n}\rd{S(\by)} \\
		& +  \frac{1}{2}\int_{\Omega_{\epsilon,R}} \Big(-\Delta N(\bx-\by)-\Delta N(\bx+\by)\Big)W(\by)\rd{\by} \\
		= & \int_{|\by|=\epsilon} \Big(N(\bx)-\frac{1}{2}N(\bx-\by)-\frac{1}{2}N(\bx+\by)\Big)\nabla W(\by)\cdot \vec{n}\rd{S(\by)} \\
		& +  \int_{|\by|=R} \Big(N(\bx)-\frac{1}{2}N(\bx-\by)-\frac{1}{2}N(\bx+\by)\Big)\nabla W(\by)\cdot \vec{n}\rd{S(\by)} \\
		&-  \frac{1}{2}\int_{|\by|=\epsilon} \Big(\nabla N(\bx-\by)-\nabla N(\bx+\by)\Big)W(\by)\cdot \vec{n}\rd{S(\by)}\\ & -  \frac{1}{2}\int_{|\by|=R} \Big(\nabla N(\bx-\by)-\nabla N(\bx+\by)\Big)W(\by)\cdot \vec{n}\rd{S(\by)} \\
		& +  \frac{1}{2}\int_{\Omega_{\epsilon,R}} \Big(-\Delta N(\bx-\by)-\Delta N(\bx+\by)\Big)W(\by)\rd{\by} \\
		=: & I_{1,\epsilon} + I_{1,R} +I_{2,\epsilon} + I_{2,R}  + I_3\,.
\end{split}\end{equation}

Then we fix $\bx\in\mathbb{R}^d\backslash \{0\}$. For sufficiently small $\epsilon>0$ and large $R>0$, we have $\bx\in \Omega_{\epsilon,R}$. Then $I_3=W(\bx)$ since $-\Delta N = \delta$ and $W$ is an even function.

{\bf STEP 1}: uniform convergence on compact sets.

Then we estimate other terms in the case $d\ne 2$. For $I_{1,\epsilon}$, we use $|\nabla W(\by)| \lesssim |\by|^{b-1}$ and the estimate
\begin{equation}
	\Big|N(\bx)-\frac{1}{2}N(\bx-\by)-\frac{1}{2}N(\bx+\by)\Big|\lesssim |\bx|^{-d}\epsilon^2,\quad \text{if }|\by| = \epsilon \le \frac{|\bx|}{2}\,,
\end{equation}
to get
\begin{equation}\label{I1epsest1}
	\left|I_{1,\epsilon}\right| \lesssim |\bx|^{-d}\epsilon^2 \cdot \epsilon^{b-1} \cdot \epsilon^{d-1} =  |\bx|^{-d}\epsilon^{b+d},\quad \text{if }\epsilon \le \frac{|\bx|}{2}\,,
\end{equation}
which converges to zero as $\epsilon\rightarrow 0$ since $b>-d$. For $I_{1,R}$, we use $|\nabla W(\by)| \lesssim |\by|^{a-1}$ and the estimate
\begin{equation}\label{NNNest1}
	\Big|N(\bx)-\frac{1}{2}N(\bx-\by)-\frac{1}{2}N(\bx+\by)\Big|\lesssim |\bx|^{2-d}+R^{2-d},\quad \text{if }|\by| = R \ge 2|\bx|\,,
\end{equation}
to get
\begin{equation}\label{I1Rest1}
	\left|I_{1,R}\right| \lesssim (|\bx|^{2-d}+R^{2-d})\cdot R^{a-1}  \cdot R^{d-1} =  |\bx|^{2-d} R^{a+d-2} + R^a,\quad \text{if }R \ge 2|\bx|\,,
\end{equation}
which converges to zero as $R\rightarrow\infty$ since $a<\min\{2-d,0\}$. Therefore we conclude that
\begin{equation}
	\lim_{\epsilon\rightarrow 0}I_{1,\epsilon} = 0,\quad \lim_{R\rightarrow\infty} I_{1,R} = 0\,.
\end{equation}
Then we estimate $I_{2,\epsilon},I_{2,R}$ similarly:
\begin{equation}\label{I1epsest2}
	\left|I_{2,\epsilon}\right| \lesssim |\bx|^{1-d}\epsilon \cdot \epsilon^b \cdot  \epsilon^{d-1} = |\bx|^{1-d}\epsilon^{b+d},\quad \text{if }\epsilon \le \frac{|\bx|}{2}\,,
\end{equation}
and
\begin{equation}\label{I1Rest2}
	\left|I_{2,R}\right| \lesssim R^{1-d} \cdot R^a \cdot R^{d-1}  =  R^a,\quad \text{if }R \ge 2|\bx|\,,
\end{equation}
which shows that
\begin{equation}
	\lim_{\epsilon\rightarrow 0}I_{2,\epsilon} = 0,\quad \lim_{R\rightarrow\infty} I_{2,R} = 0\,.
\end{equation}
This gives the limit \eqref{lem_convexpL_0} in the pointwise sense on $\mathbb{R}^d\backslash \{0\}$. Also, these estimates show that these limits are uniform on compact sets for $\bx\in\mathbb{R}^d\backslash\{0\}$.

For $d=2$, one needs to replace \eqref{NNNest1} by 
\begin{equation}
	\Big|N(\bx)-\frac{1}{2}N(\bx-\by)-\frac{1}{2}N(\bx+\by)\Big|\lesssim \big|\ln|\bx|\big|+\ln R\,,
\end{equation}
and then
\begin{equation}
	|I_{1,R}|\lesssim \big|\ln|\bx|\big| R^{a+d-2} + R^a\ln R\,,
\end{equation}
and the same convergence still works.

{\bf STEP 2}: convergence in the sense of tempered distribution.

The results \eqref{I1epsest1}\eqref{I1Rest1}\eqref{I1epsest2}\eqref{I1Rest2} in STEP 1 give
\begin{equation}
	|W_{\epsilon,R}(\bx)-W(\bx)| \lesssim |\bx|^b + |\bx|^{b+1} + |\bx|^a,\quad \text{if }2\epsilon\le |\bx|\le \frac{R}{2}\,,
\end{equation}
(up to extra logarithmic factors in $d=2$, similarly below). Then we need to estimate the size of $W_{\epsilon,R}(\bx)$ for $|\bx|<2\epsilon$ or $|\bx|>\frac{R}{2}$. If $|\bx|<2\epsilon$, then we calculate
\begin{equation}\begin{split}
		W_{\epsilon,R}(\bx) = & \int_{\Omega_{\epsilon,R}} \Big(N(\bx)-\frac{1}{2}N(\bx-\by)-\frac{1}{2}N(\bx+\by)\Big)\Delta W(\by)\rd{\by} \\
		= & \int_{\epsilon\le |\by| \le 3\epsilon} + \int_{3\epsilon< |\by| \le 1}+ \int_{1< |\by| \le R}\,.
\end{split}\end{equation}
Using $|\Delta W| \lesssim |\bx|^{b-2}$ for $0<|\bx|< 1$, we get
\begin{equation}\begin{split}
		&  \left|\int_{\epsilon\le |\by| \le 3\epsilon} \Big(N(\bx)-\frac{1}{2}N(\bx-\by)-\frac{1}{2}N(\bx+\by)\Big)\Delta W(\by)\rd{\by}\right| \\
		\lesssim & (|N(\bx)|\cdot \epsilon^d + \epsilon^2) \cdot \epsilon^{b-2} \lesssim |\bx|^{2-d}\epsilon^{b-2+d} +  \epsilon^b \lesssim |\bx|^b\,,
\end{split}\end{equation}
where the terms involving $N(\bx\pm\by)$ are handled by the estimate $\int_{|\by|\le 6\epsilon}|N(\by)|\rd{\by} \lesssim \epsilon^2$, and we used $b<\min\{2-d,0\}$ in the last inequality. The other two integrals can be estimated by
\begin{equation}\begin{split}
		& \left|\int_{3\epsilon< |\by| \le 1} \Big(N(\bx)-\frac{1}{2}N(\bx-\by)-\frac{1}{2}N(\bx+\by)\Big)\Delta W(\by)\rd{\by}\right| \\
		\lesssim & |N(\bx)|\cdot \int_{3\epsilon< |\by| \le 1}|\by|^{b-2}\rd{\by} + \int_{3\epsilon< |\by| \le 1}|\by|^{2-d}\cdot|\by|^{b-2}\rd{\by} \lesssim |\bx|^{2-d} \epsilon^{b-2+d} + \epsilon^b  \lesssim |\bx|^b\,,
\end{split}\end{equation}
using $|N(\bx\pm\by)| \lesssim |N(\by)| \lesssim |\by|^{2-d}$, and 
\begin{equation}\begin{split}
		& \left|\int_{1< |\by| \le R} \Big(N(\bx)-\frac{1}{2}N(\bx-\by)-\frac{1}{2}N(\bx+\by)\Big)\Delta W(\by)\rd{\by}\right| \\
		\lesssim & |N(\bx)|\cdot \int_{1< |\by| \le R} |\by|^{a-2}\rd{\by} + \int_{1< |\by| \le R}|\by|^{2-d}\cdot|\by|^{a-2}\rd{\by} \lesssim |\bx|^{2-d}+1\,,
\end{split}\end{equation}
where the last two integral are $O(1)$ since $a<\min\{2-d,0\}$. Therefore we get
\begin{equation}
	|W_{\epsilon,R}(\bx)|\lesssim |\bx|^b,\quad \forall |\bx|<2\epsilon\,,
\end{equation}
uniformly in $\epsilon,R$.

If $|\bx|>\frac{R}{2}$, then we calculate
\begin{equation}\begin{split}
		W_{\epsilon,R}(\bx) = & \int_{\Omega_{\epsilon,R}} \Big(N(\bx)-\frac{1}{2}N(\bx-\by)-\frac{1}{2}N(\bx+\by)\Big)\Delta W(\by)\rd{\by} \\
		= & \int_{\epsilon\le |\by| \le 1} + \int_{1< |\by| \le R/3}+ \int_{R/3< |\by| \le R}\,,
\end{split}\end{equation}
and one can use similar estimates. In fact,
\begin{equation}
	\left|\int_{\epsilon\le |\by| \le 1}\right| \lesssim \int_{\epsilon\le |\by| \le 1} |\bx|^{-d}|\by|^2 |\by|^{b-2} \rd{\by} \lesssim |\bx|^{-d} \,,
\end{equation}
\begin{equation}
	\left|\int_{1< |\by| \le R/3}\right| \lesssim |N(\bx)| \int_{1< |\by| \le R/3}|\by|^{a-2} \rd{\by} \lesssim |\bx|^{2-d} \,,
\end{equation}
\begin{equation}
	\left|\int_{R/3< |\by| \le R}\right| \lesssim R^{a-2} \int_{R/3< |\by| \le R} \Big|N(\bx)-\frac{1}{2}N(\bx-\by)-\frac{1}{2}N(\bx+\by)\Big|\rd{\by} \lesssim R^{a-2}\cdot (|\bx|^{2-d}R^d + R^2) \lesssim |\bx|^a\,,
\end{equation}
and we get
\begin{equation}
	|W_{\epsilon,R}(\bx)|\lesssim |\bx|^a,\quad \forall |\bx|>\frac{R}{2}\,,
\end{equation}
uniformly in $\epsilon,R$.

Therefore, we conclude that
\begin{equation}
	|W_{\epsilon,R}(\bx)-W(\bx)| \lesssim |\bx|^b + |\bx|^{b+1} + |\bx|^a\,,
\end{equation}
for any $\bx\ne 0$, where the implied constant is independent of $\epsilon$ and $R$. Notice that the RHS is a locally integrable function with at most polynomial growth at infinity. Therefore, combining with the uniform convergence on compact sets on $\mathbb{R}^d\backslash \{0\}$ that we have proved, we conclude that $W_{\epsilon,R}\rightarrow W$ in the sense of tempered distribution.

{\bf STEP 3}: Fourier transform.

Using the fact that 
\begin{equation}
	\hat{N}(\xi)= \frac{1}{4\pi^2}|\xi|^{-2}\,,
\end{equation}
for $\xi\ne 0$, we see that 
\begin{equation}
	\cF\Big[N(\cdot)-\frac{1}{2}N(\cdot-\by)-\frac{1}{2}N(\cdot+\by)\Big] = \frac{1}{4\pi^2}|\xi|^{-2}\big(1-\cos(2\pi\by\cdot\xi)\big)\,.
\end{equation}
This leads to 
\begin{equation}
	\hat{W}_{\epsilon,R}(\xi) = \frac{1}{4\pi^2}\int_{\epsilon \le |\by|\le R} |\xi|^{-2}\big(1-\cos(2\pi\by\cdot\xi)\big)\Delta W(\by)\rd{\by}\,.
\end{equation}
Sending $\epsilon\rightarrow0,\,R\rightarrow\infty$, since $W_{\epsilon,R}\rightarrow W$ in the sense of tempered distribution, we have $\hat{W}_{\epsilon,R}\rightarrow \hat{W}$ in the sense of tempered distribution, i.e., \eqref{lem_convexpL_1} in the sense of tempered distribution. Also, it is clear that the limit in \eqref{lem_convexpL_1} exists for any fixed $\xi\ne 0$. Using the estimate
\begin{equation}
	|\xi|^{-2}\big(1-\cos(2\pi\by\cdot\xi)\big) \lesssim \min\{|\by|^2,|\xi|^{-2}\}\,,
\end{equation}
we estimate (uniformly in $\epsilon,R$)
\begin{equation}\begin{split}
		|\hat{W}_{\epsilon,R}(\xi)|  \lesssim & \int_{\epsilon \le |\by|\le |\xi|^{-1}} |\by|^2|\Delta W(\by)| \rd{\by} + \int_{|\xi|^{-1}< |\by|\le 1} |\xi|^{-2}|\Delta W(\by)| \rd{\by} + \int_{1 < |\by|\le R} |\xi|^{-2}|\Delta W(\by)| \rd{\by} \\
		\lesssim & |\xi|^{-(b+d)} + |\xi|^{-(b-2+d)}\cdot |\xi|^{-2} + |\xi|^{-2} \lesssim |\xi|^{-(b+d)}\,,
\end{split}\end{equation}
if $|\xi| > 1$, and
\begin{equation}\begin{split}
		|\hat{W}_{\epsilon,R}(\xi)|  \lesssim & \int_{\epsilon \le |\by|\le 1} |\by|^2|\Delta W(\by)| \rd{\by} + \int_{1< |\by|\le |\xi|^{-1}} |\by|^2|\Delta W(\by)| \rd{\by} + \int_{|\xi|^{-1} < |\by|\le R} |\xi|^{-2}|\Delta W(\by)| \rd{\by} \\
		\lesssim & 1 + |\xi|^{-(a+d)} + |\xi|^{-(a-2+d)}\cdot |\xi|^{-2} \lesssim |\xi|^{-(a+d)}\,,
\end{split}\end{equation}
if $0<|\xi|\le 1$. This implies that $\hat{W}$ is a locally integrable function satisfying the same estimates.

\end{proof}

\begin{remark}
	For radial functions $W$ satisfying the assumptions of Lemma \ref{lem_convexpL}, one can integrate the $\by$ variable in \eqref{lem_convexpL_0} on spheres to simplify the representation formula. In fact, for any $s>0$, it is easy to derive that
	\begin{equation}
		\int_{|\by|=s}\Big(N(\bx)-\frac{1}{2}N(\bx-\by)-\frac{1}{2}N(\bx+\by)\Big) \rd{S(\by)} = s^{d-1}|S^{d-1}| (N(\bx) -N(s))_+\,.
	\end{equation}
	Therefore we have
	\begin{equation}\label{lem_convexpL_0r}
		W(\bx) = |S^{d-1}|\lim_{\epsilon\rightarrow0,\,R\rightarrow\infty} \int_\epsilon^R \big(N(\bx) -N(s)\big)_+ s^{d-1}\Delta W(s)\rd{s}\,.
	\end{equation}
	In particular, we see that if an essentially convex function is radial, then it is necessarily radially decreasing and positive.
	
	We remark that the strong convexity of a positive linear combination of the truncated Newtonian potentials $\big(N(\cdot) -N(s)\big)_+$ is a known result, for example, see \cite[Section 8.2]{Fug}.
\end{remark}

\begin{proof}[Proof of Lemma \ref{lem_EC}]
	
By definition, $W$ satisfies {\bf (W0)}{\bf (W1)}, and by Lemma \ref{lem_convexpL}, $\hat{W}$ is a locally integrable function. Therefore we may apply \cite[Theorem 3.4]{Shu_convex} to see that $W$ is Fourier representable at level 0, i.e.,
\begin{equation}\label{EWmu}
	E_W[\mu] = \frac{1}{2}\int_{\mathbb{R}^d\backslash\{0\}} \hat{W}(\xi)|\hat{\mu}(\xi)|^2\rd{\xi}\,,
\end{equation}
for any compactly supported signed measure $\mu$ satisfying $E_W[|\mu|]<\infty$. 

{
Consider a signed measure $\mu$ such that $E_W[\mu]$ is well-defined. If $E_W[\mu]=\infty$ then we already have $E_W[\mu]>0$. Otherwise, 
we decompose $\mu=\mu_+-\mu_-$ as its positive and negative parts., and then $E_W[\mu_+],E_W[\mu_-]$ are also finite. Then we may proceed similarly as the proof of \cite[Theorem 3.12]{Shu_convex} to see that $E_W[\mu_++\mu_-]<\infty$, i.e., $E_W[|\mu|]<\infty$. 
}

Then we truncate it as
\begin{equation}
	\mu_R := \mu\chi_{B(0;R)}\,.
\end{equation}
Then we have $E_W[\mu_R] = \frac{1}{2}\int_{\mathbb{R}^d\backslash\{0\}} \hat{W}(\xi)|\hat{\mu}_R(\xi)|^2\rd{\xi}$. Using $E_W[|\mu|]<\infty$, we may apply dominated convergence to see that 
\begin{equation}
	\lim_{R\rightarrow\infty} E_W[\mu_R] = E_W[\mu]\,.
\end{equation}
On the other hand, we have $\lim_{R\rightarrow\infty}\hat{\mu}_R(\xi)=\hat{\mu}(\xi)$ pointwisely. Thus by Fatou's lemma,
\begin{equation}
	\liminf_{R\rightarrow\infty} \frac{1}{2}\int_{\mathbb{R}^d\backslash\{0\}} \hat{W}(\xi)|\hat{\mu}_R(\xi)|^2\rd{\xi} \ge \frac{1}{2}\int_{\mathbb{R}^d\backslash\{0\}} \hat{W}(\xi)|\hat{\mu}(\xi)|^2\rd{\xi}\,.
\end{equation}
Therefore we conclude that
\begin{equation}
	E_W[\mu]\ge \frac{1}{2}\int_{\mathbb{R}^d\backslash\{0\}} \hat{W}(\xi)|\hat{\mu}(\xi)|^2\rd{\xi}\,.
\end{equation}
Due to \eqref{lem_convexpL_1} and the assumption that $\Delta W>0$ on $\mathbb{R}^d\backslash \{0\}$, we see that $\hat{W}>0$ on $\mathbb{R}^d\backslash \{0\}$, and thus $E_W[\mu]>0$. This proves the SPD property of $W$.

\end{proof}


\section{The height functional}\label{sec_H}

Let $\cS$ be a nonempty open subset of $\mathbb{R}^d$. For $\rho\in\cM(\mathbb{R}^d)$, define the \emph{height} functional
\begin{equation}\label{H}
	\cH_{\cS}[\rho] = \essinf_{\bx\in \cS} V[\rho](\bx)\,,
\end{equation}
(consistent with \eqref{H0}) which is bounded from below since $W$ and $U$ are bounded from below by {\bf (W0)}{\bf (U0)}. 

\begin{lemma}\label{lem_Hcont}
	Assume {\bf (W0)}{\bf (U0)}. Then $\cH_{\cS}$ is upper semi-continuous with respect to weak convergence on $\cM(\overline{B(0;R)})$ for any $R>0$.
\end{lemma}
\begin{proof}
	Let $\{\rho_n\}\subset \cM(\overline{B(0;R)})$ converge weakly to $\rho$. Then $V[\rho]=W*\rho+U$ is a locally integrable function, and almost all $\bx\in\mathbb{R}^d$ are its Lebesgue points, i.e., $V[\rho](\bx)<\infty$ and
	\begin{equation}
		\lim_{\epsilon\rightarrow 0}\frac{1}{|B(\bx;\epsilon)|}\int_{B(\bx;\epsilon)}V[\rho](\by)\rd{\by} = V[\rho](\bx)\,.
	\end{equation}
	Denote $H = \limsup \cH_{\cS}[\rho_n]$, and we need to prove that $\cH_{\cS}[\rho] \ge H$. 
	
	Fix a mollifier $\phi_\epsilon$. For any Lebesgue point $\bx\in \cS$ of $V[\rho]$, suppose $V[\rho](\bx) < H-\alpha$ for some $\alpha>0$, then there exists $\epsilon>0$ such that $B(\bx;\epsilon)\subset \cS$ and
	\begin{equation}
		(V[\rho]*\phi_\epsilon)(\bx)<H-\alpha\,.
	\end{equation}
	Notice that
	\begin{equation}
		V[\rho]*\phi_\epsilon = (W*\phi_\epsilon)*\rho + U*\phi_\epsilon\,,
	\end{equation}
	and $W*\phi_\epsilon$ is continuous. Since $\rho_n$ converges weakly to $\rho$ with each $\rho_n$ supported on  $\overline{B(0;R)}$, we see that
	\begin{equation}
		\lim_{n\rightarrow\infty}(V[\rho_n]*\phi_\epsilon)(\bx) = (V[\rho]*\phi_\epsilon)(\bx) < H-\alpha\,,
	\end{equation}
	and thus for sufficiently large $n$
	\begin{equation}
		(V[\rho_n]*\phi_\epsilon)(\bx) < H-\alpha\,.
	\end{equation}
	Since $B(\bx;\epsilon)\subset \cS$, this contradicts $\cH_{\cS}[\rho_n] = \essinf_{\cS} V[\rho_n] > H-\frac{\alpha}{2}$ for sufficiently large $n$. Therefore we conclude that $V[\rho](\bx) \ge H$ for any Lebesgue point $\bx$ of $V[\rho]$. Since almost all $\bx\in \cS$ are Lebesgue points of $V[\rho]$, we get $\cH_{\cS}[\rho] \ge H$.

\end{proof}

Then we prove the existence of maximizer for $\cH_\cS$ under a subcritical condition.
\begin{theorem}[Existence of maximizer]\label{thm_Hexist}
	Assume {\bf (W0)}{\bf (U0)} and $\lim_{|\bx|\rightarrow\infty}U(\bx)=U_\infty\in\mathbb{R}$. Assume $W$ is not identically 0, $W\ge 0$ and $\lim_{|\bx|\rightarrow\infty}W(\bx)=0$. Assume 
	\begin{equation}\label{thm_Hexist_1}
		\sup_{\rho\in\cM(\mathbb{R}^d)}\cH_{\cS}[\rho] < U_\infty\,.
	\end{equation}
	Then there exists a maximizer of $\cH_{\cS}$ in $\cM(\mathbb{R}^d)$, and any maximizer of $\cH_{\cS}$ in $\cM(\mathbb{R}^d)$ is compactly supported.
\end{theorem}

We remark that the assumption \eqref{thm_Hexist_1} can indeed be dropped when $\cS$ is bounded. This is because $\cH_{\cS}[\rho]$ is independent of the values of $U$ on $\cS^c$, one can justify $\sup_{\rho\in\cM(\mathbb{R}^d)}\cH_{\cS}[\rho]<\infty$, and one can modify $U$ on $\cS^c$ so that \eqref{thm_Hexist_1} becomes true. The detail of such an argument can be found in the proof of Lemma \ref{lem_rhosharp} in the next section.

\begin{proof}
	We start by denoting 
	\begin{equation}
		\epsilon := U_\infty - \sup_{\rho\in\cM(\mathbb{R}^d)}\cH_{\cS}[\rho]  > 0\,.
	\end{equation}
	We take $R$ sufficiently large so that
	\begin{equation}
		U(\bx)\ge U_\infty-\frac{\epsilon}{2},\quad \forall |\bx|\ge R\,.
	\end{equation} 
	Then, for any $\rho\in\cM(\mathbb{R}^d)$, we have 
	\begin{equation}
		V[\rho](\bx)\ge U_\infty-\frac{\epsilon}{2},\quad \forall |\bx|\ge R\,,
	\end{equation}
	since $W\ge 0$.
	
	Since $W$ is not identically 0, $W\ge 0$ and $W$ is lower-semicontinuous, we have $W>0$ on some open set. Therefore, by further increasing $R$, we have
	\begin{equation}
		\int_{B(0;2R)}W(\bx-\by)\rd{\by}\ge \alpha > 0,\quad \forall \bx\in B(0;R)\,,
	\end{equation}
	and
	\begin{equation}
		B(0;R)\cap \cS \ne \emptyset\,.
	\end{equation}
	Since $\lim_{|\bx|\rightarrow\infty}W(\bx)=0$, we take $R_1>2R$ so that
	\begin{equation}
		0\le W(\bx) \le \frac{\alpha}{2|B(0;2R)|},\quad \forall |\bx| \ge R_1-R\,.
	\end{equation}
	
	Let $\{\rho_n\}$ be a maximizing sequence of $\cH_{\cS}$. Then for each $\rho_n$, define
	\begin{equation}\label{tilderho}
		\tilde{\rho}_n = \rho_n\chi_{B(0;R_1)} + \big(1-\rho_n(B(0;R_1))\big)\frac{1}{|B(0;2R)|}\chi_{B(0;2R)} \in \cM(\mathbb{R}^d)\,.
	\end{equation}
	Then, for any $\bx\in B(0;R)$,
	\begin{equation}\begin{split}
			V[\tilde{\rho}_n](\bx) - V[\rho_n](\bx) = & \int_{\mathbb{R}^d}W(\bx-\by)\Big(\big(1-\rho_n(B(0;R_1))\big)\frac{1}{|B(0;2R)|}\chi_{B(0;2R)} - \rho_n\chi_{B(0;R_1)^c} \Big)\rd{\by} \\
			= & \big(1-\rho_n(B(0;R_1))\big)\frac{1}{|B(0;2R)|}\int_{B(0;2R)}W(\bx-\by)\rd{\by} - \int_{B(0;R_1)^c}W(\bx-\by)\rho_n(\by)\rd{\by} \\
			\ge & \big(1-\rho_n(B(0;R_1))\big)\frac{\alpha}{|B(0;2R)|} - \frac{\alpha}{2|B(0;2R)|}\int_{B(0;R_1)^c}\rho_n(\by)\rd{\by} \\
			= & \big(1-\rho_n(B(0;R_1))\big)\frac{\alpha}{2|B(0;2R)|} \ge 0\,,
	\end{split}\end{equation}
	where in the first inequality we used that $|\bx-\by|\ge R_1-R$ in the second integral. This implies 
	\begin{equation}\begin{split}
			\cH_{\cS}[\tilde{\rho}_n] = & \min\Big\{\essinf_{\bx\in B(0;R)\cap \cS} V[\tilde{\rho}_n](\bx), \essinf_{\bx\in B(0;R)^c \cap \cS} V[\tilde{\rho}_n](\bx)\Big\} \\
			\ge & \min\big\{\cH_{\cS}[\rho_n], U_\infty-\frac{\epsilon}{2}\big\} = \cH_{\cS}[\rho_n]\,, \\
	\end{split}\end{equation}
	i.e., $\{\tilde{\rho}_n\}$ is also a maximizing sequence of $\cH_{\cS}$. 
	
	Since every $\tilde{\rho}_n$ is supported inside $\overline{B(0;R_1)}$, we may take a weakly convergent subsequence, still denoted as $\{\tilde{\rho}_n\}$, which converges weakly to $\rho_\infty\in \cM(\mathbb{R}^d)$. Then Lemma \ref{lem_Hcont} shows that $\rho_\infty$ is a maximizer of $\cH_{\cS}$.
	
	To see that any maximizer $\rho$ of $\cH_{\cS}$ is compactly supported, we define $\tilde{\rho}$ as in \eqref{tilderho}, and one can see that if $\rho(B(0;R_1))<1$ then $V[\tilde{\rho}](\bx) - V[\rho](\bx) \ge (1-\rho(B(0;R_1)))\frac{\alpha}{2|B(0;2R)|} > 0$ for any $\bx\in  B(0;R)$. Since $B(0;R)\cap \cS \ne \emptyset$, this would lead to $\cH_{\cS}[\tilde{\rho}] > \cH_{\cS}[\rho]$, a contradiction. Therefore any maximizer $\rho$ of $\cH_{\cS}$ is supported in $\overline{B(0;R_1)}$.
	
\end{proof}

Then we derive the Euler-Lagrange condition for maximizers of $\cH_\cS$, under the extra assumption that $\Delta W>0$ on $\mathbb{R}^d\backslash \{0\}$.

\begin{theorem}[Euler-Lagrange condition]\label{thm_ELH}
	Under the same assumptions as Theorem \ref{thm_Hexist}, further assume $W\in C^2(\mathbb{R}^d\backslash\{0\})$ with $\Delta W(\bx) > 0,\,\forall \bx\ne 0$.  If $\rho$ is a maximizer of $\cH_{\cS}$ in $\cM(\mathbb{R}^d)$, then $\supp\rho\subset \bar{\cS}$, and 
	\begin{equation}\label{thm_ELH_1}
		V[\rho] \le C_0,\quad \textnormal{on }\supp\rho\,,
	\end{equation}
	\begin{equation}\label{thm_ELH_2}
		V[\rho] \ge C_0,\quad \textnormal{a.e. }\cS\,,
	\end{equation}
	are satisfied for some $C_0\in\mathbb{R}$.
\end{theorem}

To prove this theorem, we start from a simple calculation for positive-Laplacian functions.
\begin{lemma}\label{lem_posDelta}
	Let $u$ be a $C^2$ function in a neighborhood of $\overline{B(\bx;\delta)}$ with $\Delta u \ge \epsilon>0$ on $\overline{B(\bx;\delta)}$. Then
	\begin{equation}
		\frac{1}{|B(\bx;\delta)|}\int_{B(\bx;\delta)}u(\by)\rd{\by} - u(\bx) \gtrsim \epsilon \delta^2\,,
	\end{equation}
	where the implied constant only depends on $d$.
\end{lemma}
\begin{proof}
	Define 
	\begin{equation}
		\Phi(s)  = \frac{1}{|\partial B(0;1)|}\int_{\partial B(0;1)}u(\bx+s\by)\rd{S(\by)}\,,
	\end{equation}
	for $0<s\le \delta$. Then $\lim_{s\rightarrow 0^+} \Phi(s) = u(\bx)$, and
	\begin{equation}\begin{split}
			\Phi'(s) = & \frac{1}{|\partial B(0;1)|}\int_{\partial B(0;1)}\nabla u(\bx+s\by)\cdot \by\rd{S(\by)} = \frac{1}{|\partial B(0;1)|}\int_{\partial B(0;1)}\nabla u(\bx+s\by)\cdot \vec{n}\rd{S(\by)} \\
			= & \frac{1}{|\partial B(0;1)|}\int_{B(0;1)}s \Delta u(\bx+s\by) \rd{\by} \gtrsim s\epsilon\,,
	\end{split}\end{equation}
	for $0<s< \delta$. This implies
	\begin{equation}
		\Phi(s) - u(\bx) \gtrsim s^2\epsilon\,.
	\end{equation}
	Therefore
	\begin{equation}\begin{split}
			\frac{1}{|B(\bx;\delta)|}\int_{B(\bx;\delta)}u(\by)\rd{\by} - u(\bx)= & \frac{1}{|B(\bx;\delta)|}\int_{B(\bx;\delta)}(u(\by) - u(\bx))\rd{\by}\\
			= & \frac{1}{|B(\bx;\delta)|}\int_0^\delta \int_{\partial B(0;1)}\big(u(\bx+r\by)-u(\bx)\big)\rd{S(\by)} r^{d-1}\rd{r} \\
			= & \frac{|\partial B(0;1)|}{|B(\bx;\delta)|}\int_0^\delta \big(\Phi(r)-u(\bx)\big)r^{d-1}\rd{r}  \\
			\gtrsim & \delta^{-d}\int_0^\delta r^2\epsilon r^{d-1}\rd{r} \gtrsim \epsilon \delta^2\,.
	\end{split}\end{equation}
	
\end{proof}

Then we introduce a local operation, called `microscopic diffusion' in \cite[Lemma 2.4]{SW21}, which distributes any mass point uniformly into a ball of radius $\delta$. We show that such an operation always increases the generated potential away from the support.
\begin{lemma}\label{lem_posDelta2}
	Assume $W\in C^2(\mathbb{R}^d\backslash\{0\})$ with $\Delta W(\bx) > 0,\,\forall \bx\ne 0$. Let $\sigma$ be a nonnegative measure supported in $\overline{B(\bx;\delta)}$, and define
	\begin{equation}
		\mu = -\sigma + \sigma * \frac{\chi_{B(0;\delta)}}{|B(0;\delta)|}\,,
	\end{equation}
	which is a mean-zero signed measure supported in $\overline{B(\bx;2\delta)}$. Then 
	\begin{equation}
		(W*\mu)(\by) \gtrsim   |\sigma|\delta^2\inf_{B(0;|\by-\bx|+2\delta)\backslash B(0;\delta)}\Delta W ,\quad \forall \by \textnormal{ with } |\by-\bx|\ge 3\delta\,,
	\end{equation}
	where the implied constant only depends on $d$.
\end{lemma}

\begin{proof}
	Take $\by$ with $|\by-\bx|\ge 3\delta$. Then
	\begin{equation}\begin{split}
			(W*\mu)(\by) = & -(W*\sigma)(\by) + \Big(\big(W*\frac{\chi_{B(0;\delta)}}{|B(0;\delta)|}\big)*\sigma\Big)(\by) \\
			= & \Big(\big(W*\frac{\chi_{B(0;\delta)}}{|B(0;\delta)|} - W\big) * \sigma\Big)(\by) \\
			= & \int_{\by-\overline{B(\bx;\delta)}}\Big(W*\frac{\chi_{B(0;\delta)}}{|B(0;\delta)|} - W\Big)(\bz)\sigma(\by-\bz)\rd{\bz}\,. \\
	\end{split}\end{equation}	
	Notice that for any $\bz\in \by-\overline{B(\bx;\delta)}$, we have $|\bz|\ge 2\delta$, and thus applying Lemma \ref{lem_posDelta} gives
	\begin{equation}
		\Big(W*\frac{\chi_{B(0;\delta)}}{|B(0;\delta)|}\Big)(\bz) - W(\bz) \gtrsim \delta^2\inf_{B(\bz;\delta)}\Delta W \,.
	\end{equation}
	Any point $\bz'$ in such $B(\bz;\delta)$ satisfies $|\bz'|<|\by-\bx|+2\delta$ and $|\bz'|>\delta$. Therefore, integrating the above inequality against $\sigma(\by-\bz)\rd{\bz}$ gives the conclusion.
\end{proof}

\begin{proof}[Proof of Theorem \ref{thm_ELH}]
	Without loss of generality, assume $\inf U=0$.
	
	Theorem \ref{thm_Hexist} gives that $\rho$ is compactly supported. By the assumptions $\lim_{|\bx|\rightarrow\infty}U(\bx)=U_\infty\in\mathbb{R}$ and $\lim_{|\bx|\rightarrow\infty}W(\bx)=0$, we have
	\begin{equation}
		\lim_{|\bx|\rightarrow\infty}V[\rho](\bx)=U_\infty\,.
	\end{equation} 
	Combining with the assumption $\cH_{\cS}[\rho] = \max_{\cM(\mathbb{R}^d)} \cH_\cS[\cdot] < U_\infty$, we see that the set
	\begin{equation}
		T_1:=\Big\{\by\in\cS: V[\rho](\by)<\frac{1}{2}(\cH_\cS[\rho]+U_\infty)\Big\}\,,
	\end{equation}
	has positive Lebesgue measure and compact support.
	
	We first prove $\supp\rho\subset \bar{\cS}$. Suppose not, then there exists $\bx$ such that $B(\bx;3\delta)\cap \cS=\emptyset$ for some $\delta>0$ and $\rho(B(\bx;\delta))>0$. Then define a probability measure
	\begin{equation}
		\tilde{\rho} = \rho -  \sigma + \sigma * \frac{\chi_{B(0;\delta)}}{|B(0;\delta)|},\quad \sigma = \rho\chi_{B(\bx;\delta)}\,.
	\end{equation}
	We may apply Lemma \ref{lem_posDelta2} to get that
	\begin{equation}
		V[\tilde{\rho}](\by)-V[\rho](\by) \gtrsim  \rho(B(\bx;\delta))\delta^2\inf_{B(0;|\by-\bx|+2\delta)\backslash B(0;\delta)}\Delta W\,,
	\end{equation}
	for any $\by$ with $|\by-\bx|\ge 3\delta$. In particular, this works for any $\by\in \cS$. Also, since $T_1\subset \cS$ has compact support, we get 
	\begin{equation}
		V[\tilde{\rho}](\by)-V[\rho](\by) \gtrsim  \rho(B(\bx;\delta))\delta^2\,,
	\end{equation}
	for any $\by\in T_1$, since the infimum of $\Delta W$ is taken on a compact subset of $\mathbb{R}^d\backslash \{0\}$ and thus strictly positive. Therefore we get $\cH_\cS[\tilde{\rho}] > \cH_\cS[\rho]$ which is a contradiction.
	
	Then we only need to prove that for any $\bx\in\supp\rho$, 
	\begin{equation} 
		|\{\by\in\cS: V[\rho](\by)<V[\rho](\bx)\}| = 0\,.
	\end{equation}
	In fact, if this is true, then we take $C_0=\essinf_{\cS} V[\rho]$ which satisfies \eqref{thm_ELH_1}\eqref{thm_ELH_2}.
	
	Suppose this is false for some $\bx\in\supp\rho$. Then there exists $\epsilon>0$ such that
	\begin{equation}
		T_2:=\{\by\in\cS: V[\rho](\by)<V[\rho](\bx)-\epsilon\}\,,
	\end{equation}
	has positive Lebesgue measure. This implies that 
	\begin{equation}\label{Vrhoeps}
		V[\rho](\bx) > \epsilon\,,
	\end{equation}
	because $V[\rho]$ is nonnegative due to the assumptions $W\ge 0$ (assumed in Theorem \ref{thm_Hexist}) and $U\ge 0$ (since $\inf U=0$ is assumed at the beginning of the proof).  It is clear that one of $T_1$ or $T_2$ contains the other, and thus 
	\begin{equation}
		T:=T_1\cap T_2 = \Big\{\by\in\cS: V[\rho](\by)<\min\big\{V[\rho](\bx)-\epsilon,\frac{1}{2}(\cH_\cS[\rho]+U_\infty)\big\}\Big\}\,,
	\end{equation}
	has positive Lebesgue measure and compact support.

	Since $V[\rho]$ is lower-semicontinuous, there exists $\delta>0$ such that 
	\begin{equation}
		V[\rho](\by)>V[\rho](\bx)-\frac{\epsilon}{4},\quad \forall \by\in B(\bx;3\delta)\,.
	\end{equation}
	In particular, $B(\bx;\delta)\cap T =\emptyset$. 
	
	Since $\bx\in\supp\rho$, we have $\rho(B(\bx;\delta))>0$. Then we define a probability measure
	\begin{equation}
		\tilde{\rho} = \rho -  \sigma + \sigma * \frac{\chi_{B(0;\delta)}}{|B(0;\delta)|},\quad \sigma = \alpha\rho\chi_{B(\bx;\delta)}\,,
	\end{equation}
	where $0<\alpha<1$ is a parameter to be determined.
	
	First, for any $\by\in B(\bx;3\delta)$, using  $W\ge 0$  and $U\ge 0$, we have
	\begin{equation}
		V[\tilde{\rho}](\by) \ge V[\rho](\by) - \alpha \big(W*(\rho\chi_{B(\bx;\delta)})\big)(\by) \ge (1-\alpha)V[\rho](\by) > (1-\alpha)\big(V[\rho](\bx)-\frac{\epsilon}{4}\big)\,,
	\end{equation}
	which is greater than $V[\rho](\bx)-\frac{\epsilon}{2}$ if we choose
	\begin{equation}
		\alpha = \frac{\frac{\epsilon}{4}}{V[\rho](\bx)-\frac{\epsilon}{4}}\,,
	\end{equation}
	(where the denominator is at least $\frac{3\epsilon}{4}$ due to \eqref{Vrhoeps}).
	
	Then, for any $\by\notin B(\bx;3\delta)$, we may apply Lemma \ref{lem_posDelta2}, with the compact support property of $T$, to get
	\begin{equation}
		V[\tilde{\rho}](\by)-V[\rho](\by) > 0\,,
	\end{equation}
	for any $\by\notin B(\bx;3\delta)$, and
	\begin{equation}
		V[\tilde{\rho}](\by)-V[\rho](\by) \gtrsim  \rho(B(\bx;\delta))\delta^2\,,
	\end{equation}
	for any $\by\in T$. Therefore we get $\cH_\cS[\tilde{\rho}] > \cH_\cS[\rho]$ which is a contradiction.
\end{proof}

\section{Proof of the main results}\label{sec_main}

\begin{lemma}\label{lem_rhosharp}
	Assume $W$ satisfies the assumptions in Theorems \ref{thm_Hexist} and \ref{thm_ELH}, and $W$ is CSPD. Assume $\rho_\sharp\in\cM(\mathbb{R}^d)$ is a compactly supported locally integrable function, $\supp\rho_\sharp=\bar{\cS}$ for some open set $\cS$ with $\partial \cS$ having zero Lebesgue measure, and $W*\rho_\sharp$ is continuous on $\mathbb{R}^d$. Then, for any $\rho\in\cM(\mathbb{R}^d)$, we have
	\begin{equation}
		\essinf_{\cS}(W*\rho-W*\rho_\sharp) \le 0\,.
	\end{equation}
\end{lemma}

\begin{proof}
	We first claim that $\essinf_{\cS}W*\rho$ is bounded from above, say, by $C_\sharp$ for any $\rho\in\cM(\mathbb{R}^d)$. In fact, since $\bar{\cS}$ is compact, $W$ is locally integrable with $\lim_{|\bx|\rightarrow\infty}W(\bx)=0$, we see that  $\int_{\cS}W(\bx-\by)\rd{\bx} \le C_W$ for any $\by\in\mathbb{R}^d$ where $C_W>0$ is independent of $\by$. Therefore, integrating against $\rd{\rho(\by)}$, we get $\int_{\cS}(W*\rho)(\bx)\rd{\bx} \le C_W$, and it follows that $\essinf_{\cS}(W*\rho)\le \frac{C_W}{|\cS|}=:C_\sharp$.
	
	Define
	\begin{equation}
		U_\sharp(\bx) = \left\{\begin{split}
			& -(W*\rho_\sharp)(\bx),\quad \bx\in\bar{\cS} \\
			& C_\sharp + 1 ,\quad \text{for sufficiently large }|\bx| \\
			& \text{extended continuously},\quad \text{otherwise}
		\end{split}\right.\,,
	\end{equation}
	which is continuous and satisfies $\lim_{|\bx|\rightarrow\infty}U_\sharp(\bx) = C_\sharp + 1$. Denote
	\begin{equation}
		V_\sharp[\rho] = W*\rho + U_\sharp,\quad \cH_{\cS,\sharp}[\rho] = \essinf_{\cS} V_\sharp[\rho] \,.
	\end{equation} 
	We verify that
	\begin{equation}\begin{split}
		\sup_{\rho\in\cM(\mathbb{R}^d)} \cH_{\cS,\sharp}[\rho] = & 	\sup_{\rho\in\cM(\mathbb{R}^d)} \essinf_{\cS} (W*\rho + U_\sharp) = \sup_{\rho\in\cM(\mathbb{R}^d)} \essinf_{\cS} (W*\rho -W*\rho_\sharp)  \\
		\le & \sup_{\rho\in\cM(\mathbb{R}^d)}\essinf_{\cS}(W*\rho) \le C_\sharp\,.
	\end{split}\end{equation}
	Therefore we may apply Theorem \ref{thm_Hexist} with external potential $U_\sharp$ to get the existence of a maximizer of $\cH_{\cS,\sharp}$, call it $\rho_0$. Applying Theorem \ref{thm_ELH} gives
	\begin{equation}\label{ELrho0}
		\supp\rho_0\subset \bar{\cS}\,; \qquad V_\sharp[\rho_0]\le C_0,\quad \textnormal{on }\supp\rho_0\,;\qquad V_\sharp[\rho_0]\ge C_0,\quad \textnormal{a.e. }\cS\,,
	\end{equation}
	for some $C_0$.
	
	We claim that $\rho_0=\rho_\sharp$. In fact, since $V_\sharp[\rho_\sharp]=0$ on $\bar{\cS}$ by definition, Lemma \ref{lem_suf} shows that $\rho_\sharp$ is the unique minimizer of $E_{W,U_\sharp}$ in $\cM(\bar{\cS})$. Then denote the linear interpolation between $\rho_0$ and $\rho_\sharp$ as
	\begin{equation}
		\rho_{t} = (1-t)\rho_0 + t \rho_\sharp\,,
	\end{equation}
	for $0\le t \le 1$. Then, since $\supp\rho_0\subset \bar{\cS}$,
	\begin{equation}
		\frac{\rd}{\rd{t}}\Big|_{t=0}E_{W,U_\sharp}[\rho_{t}] = \int_{\bar{\cS}} V_\sharp[\rho_0]\rho_\sharp\rd{\bx} - \int_{\bar{\cS}} V_\sharp[\rho_0]\rho_0\rd{\bx}\,.
	\end{equation}
	Since $V_\sharp[\rho_0]\le C_0$ on $\supp\rho$, we see that $\int_{\bar{\cS}} V_\sharp[\rho_0]\rho_0\rd{\bx} \le C_0$.  Since $V_\sharp[\rho_0]\ge C_0,\, \textnormal{a.e. }\cS$ and $\partial S$ has zero Lebesgue measure, we get $V_\sharp[\rho_0]\ge C_0,\, \textnormal{a.e. }\bar{\cS}$. Then we get $\int_{\bar{\cS}} V_\sharp[\rho_0]\rho_\sharp\rd{\bx} \ge C_0$ since $\rho_\sharp$ is assumed to be a locally integrable function. Therefore we get 
	\begin{equation}
		\frac{\rd}{\rd{t}}\Big|_{t=0}E_{W,U_\sharp}[\rho_{t}] \ge 0\,.
	\end{equation}
	Due to the CSPD property of $W$ and Lemma \ref{lem_CSPD}, if $\rho_0\ne \rho_\sharp$, we would get $E_{W,U_\sharp}[\rho_\sharp] > E_{W,U_\sharp}[\rho_0]$, contradicting the fact that $\rho_\sharp$ is the unique minimizer of $E_{W,U_\sharp}$. Therefore $\rho_0=\rho_\sharp$. 
	
	As a consequence, for any $\rho\in\cM(\mathbb{R}^d)$, we have $\cH_{\cS,\sharp}[\rho] \le \cH_{\cS,\sharp}[\rho_\sharp] = 0$, i.e.,
	\begin{equation}
		\essinf_{\cS}(W*\rho-W*\rho_\sharp) \le 0\,.
	\end{equation}
	
\end{proof}

\begin{remark}
	We point out a subtlety in the part of the above proof concerning $\rho_0=\rho_\sharp$. Generally speaking, when $W$ is CSPD, if $D$ is a closed subset of $\mathbb{R}^d$, a condition like 
	\begin{equation}
		V[\rho_0]\le C_0,\quad \textnormal{on }\supp\rho_0\,;\qquad V[\rho_0]\ge C_0,\quad \textnormal{a.e. }D\,,
	\end{equation}
	for a probability measure $\rho_0\in\cM(D)$ may not be sufficient to guarantee that $\rho_0$ is the minimizer of $E$ in $\cM(D)$. In fact, when doing a direct linear interpolation argument with the unique minimizer (denoted $\rho_\infty$), one may not reach a contradiction if $\rho_\infty$ has positive mass in the set $\{V[\rho_0]<C_0\}$ which has Lebesgue measure zero. A mollification argument like the proof of \cite[Theorem 2.4]{CS21} does not apply in general, due to the domain constraint. In the proof of Lemma \ref{lem_rhosharp} we can derive $\rho_0=\rho_\sharp$ as the energy minimizer because we already know that the minimizer $\rho_\sharp$ is a locally integrable function.
\end{remark}

\begin{proof}[Proof of Theorem \ref{thm_main}]
	Let $D$ be the underlying domain as described in Theorem \ref{thm_main}. Lemma \ref{lem_convexpL} gives the SPD property of $W$. Then $W$ satisfies all the assumptions of Lemma \ref{lem_rhosharp} (not identically 0, $W\ge 0$, $\lim_{|\bx|\rightarrow\infty}W(\bx)=0$, $W\in C^2(\mathbb{R}^d\backslash\{0\})$ with $\Delta W(\bx) > 0,\,\forall \bx\ne 0$, $W$ is CSPD), and so does $\rho_\sharp$. We apply Lemma \ref{lem_rhosharp} to get
	\begin{equation}
		\essinf_{\cS}(V[\rho]-V[\rho_\sharp]) = \essinf_{\cS}(W*\rho-W*\rho_\sharp)\le 0\,,
	\end{equation}
	for any $\rho\in\cM(\mathbb{R}^d)$. Combining with \eqref{thm_main_1}, we see that
	\begin{equation}
		\cH_\cS[\rho]=\essinf_{\cS}V[\rho] \le \sup_{\cS}V[\rho_\sharp] < U_\infty\,,
	\end{equation}
	for any $\rho\in\cM(\mathbb{R}^d)$. Then, denoting $\cD$ as the interior of $D$, we have $\cS\subset \cD$ and thus
	\begin{equation}
		\sup_{\rho\in\cM(\mathbb{R}^d)} \cH_{\cD}[\rho] \le \sup_{\rho\in\cM(\mathbb{R}^d)} \cH_{\cS}[\rho] < U_\infty\,.
	\end{equation}
	In all the assumed cases, we have $D=\bar{\cD}$, and $D\backslash \cD$ has zero Lebesgue measure.
	
	Then, applying Theorem \ref{thm_Hexist} (with the $\cS$ therein being $\cD$) gives that there exists a maximizer $\rho_\infty$ of $\cH_{\cD}$ in $\cM(\mathbb{R}^d)$ which is compactly supported.
	
	Finally, applying Theorem \ref{thm_ELH} gives that $\supp\rho_\infty\subset \bar{\cD}=D$, and
	\begin{equation}
		V[\rho_\infty] \le C_0,\quad \textnormal{on }\supp\rho_\infty,\qquad V[\rho_\infty] \ge C_0,\quad \textnormal{a.e. }\cD\,.
	\end{equation}
	We further have $V[\rho_\infty] \ge C_0,\, \textnormal{a.e. }D$ since $D\backslash \cD$ has zero Lebesgue measure. Then, in view of the {\bf (W1)} assumption and the continuity of $U$, we may apply Lemma \ref{lem_suf} (for the case $D=\mathbb{R}^d$) or Lemma \ref{lem_sufhalf} below (for the other cases) to get that $\rho_\infty$ is the unique minimizer of $E_{W,U}$ in $\cM(D)$.
	
\end{proof}

\begin{lemma}\label{lem_sufhalf}
	Under the same assumptions as Lemma \ref{lem_SCuni}, further assume {\bf (W1)} and $U$ is continuous. Suppose $D$ is given as
	\begin{itemize}
		\item $D=[x_0,\infty)$ for some $x_0\in\mathbb{R}$, if $d=1$;
		\item $D=\{\bx=(x_1,\dots,x_d)\in\mathbb{R}^d: x_d \ge \Phi(\hat{\bx})\}$, where $\hat{\bx}:=(x_1,\dots,x_{d-1})$ and $\Phi:\mathbb{R}^{d-1}\rightarrow \mathbb{R}$ is a continuous function, if $d=2$.
	\end{itemize}
	Then \eqref{lem_EL_1} implies that $\rho$ is the unique minimizer of $E_{W,U}$ in $\cM(D)$.
\end{lemma}

\begin{proof}
	We will treat the case $d\ge 2$, and the case $d=1$ can be handled similarly.
	
	Assume the contrary that $\rho_1\in\cM(D)$ satisfies $E_{W,U}[\rho_1] < E_{W,U}[\rho]$. By truncation, we may assume without loss of generality that $\rho_1$ is compactly supported. Since $U$ is continuous in a neighborhood of $S$ and $\Phi$ is continuous, we may take a modulus of continuity $\omega(\epsilon)$, satisfying $\lim_{\epsilon\rightarrow0}\omega(\epsilon)=0$ and
	\begin{equation}\label{moc}
		|U(\bx)-U(\by)| \le \omega(|\bx-\by|),\quad |\Phi(\hat{\bx})
		-\Phi(\hat{\by})| \le \omega(|\hat{\bx}-\hat{\by}|)\,,
	\end{equation}
	whenever $\dist(\bx,\supp\rho_1)$ and $\dist(\by,\supp\rho_1)$ are sufficiently small. We denote
	\begin{equation}
		\alpha := E_{W,U}[\rho] - E_{W,U}[\rho_1] > 0\,.
	\end{equation}
	
	Let $\phi_\epsilon$ be a mollifier with $\epsilon>0$ small, and denote $\rho_{1,\epsilon}=\rho_1*\phi_\epsilon$. Similar to the proof of Lemma \ref{lem_suf}, we get that 
	\begin{equation}\label{alp3}
		E_{W,U}[\rho_{1,\epsilon}] \le E_{W,U}[\rho_1] + \frac{\alpha}{3}\,,
	\end{equation}
	for sufficiently small $\epsilon$. Then we claim that for sufficiently small $\epsilon$ we have
	\begin{equation}
		\supp\rho_{1,\epsilon} \subset D_{\epsilon+\omega(\epsilon)}\,,
	\end{equation}
	where
	\begin{equation}
		D_{\beta}:=\{\bx=(x_1,\dots,x_d)\in\mathbb{R}^d: x_d \ge \Phi(\hat{\bx})-\beta\}\,,
	\end{equation}
	is an expansion of $D$ in the $x_d$ direction. To see this, we notice that any point $\bx\in \supp\rho_{1,\epsilon}$ satisfies that $\dist(\bx,\supp\rho_1) \le \epsilon$. Therefore \eqref{moc} applies to any $\bx,\by\in\supp\rho_{1,\epsilon}$ when $\epsilon$ is small. For any $\bx\in\supp\rho_{1,\epsilon}$, we can find $\by\in\supp\rho_1\subset D$ with $|\bx-\by|\le \epsilon$, which leads to
	\begin{equation}
		|\hat{\bx}-\hat{\by}|^2 + (x_d-y_d)^2 \le \epsilon^2,\quad y_d \ge \Phi(\hat{\by})\,.
	\end{equation}
	Then we see that $|\hat{\bx}-\hat{\by}| \le \epsilon$, which implies $|\Phi(\hat{\bx})
	-\Phi(\hat{\by})| \le \omega(\epsilon)$ by \eqref{moc}. We also have $|x_d-y_d|\le \epsilon$, and thus
	\begin{equation}
		x_d \ge y_d-\epsilon \ge  \Phi(\hat{\by})-\epsilon \ge \Phi(\hat{\bx}) -  \omega(\epsilon)-\epsilon\,,
	\end{equation}
	and the claim is proved.
	
	The claim implies that 
	\begin{equation}
		\rho_{2,\epsilon}(\hat{\bx},x_d):=\rho_{1,\epsilon}\big(\hat{\bx},x_d-(\epsilon+\omega(\epsilon))\big)\,,
	\end{equation}
	satisfy that $\supp\rho_{2,\epsilon}\subset D$. Also, $\rho_{2,\epsilon}$ is a translation of $\rho_{1,\epsilon}$, the $W$ parts of their energies are the same. Then, using \eqref{moc} for the $U$ parts of their energies, we have the estimate
	\begin{equation}
		|E_{W,U}[\rho_{2,\epsilon}]-E_{W,U}[\rho_{1,\epsilon}]| \le \omega\big(\epsilon+\omega(\epsilon)\big)\,.
	\end{equation}
	Since $\lim_{\epsilon\rightarrow0}\omega(\epsilon)=0$, the last quantity can be made less than $\frac{\alpha}{3}$ for sufficiently small $\epsilon$. Combining with \eqref{alp3}, this gives $E_{W,U}[\rho_{2,\epsilon}] < E_{W,U}[\rho]$. Then we may proceed as in the proof of Lemma \ref{lem_suf}, with $\rho_{2,\epsilon}$ in place of $\tilde{\rho}_1$, to get the conclusion.
	
\end{proof} 

\begin{proof}[Proof of Proposition \ref{prop_nece}]
	First, since $W*\rho_\infty$ and $U$ are continuous, we see that $V[\rho_\infty]=W*\rho_\infty + U$ is continuous. Therefore the Euler-Lagrange condition gives that $V[\rho_\infty]=C_0$ on $\supp\rho_\infty$. It is easy to see that $\liminf_{|\bx|\rightarrow\infty}V[\rho](\bx) = U_\infty$ for any $\rho\in\cM(\mathbb{R}^d)$. Then, by the Euler-Lagrange condition and the assumption \eqref{prop_nece_1}, we see that $\rho_\infty$ is compactly supported.
	
	Let $\phi_\epsilon$ be a mollifier, and define $\rho_{\infty,\epsilon}=\rho_\infty*\phi_\epsilon$. Then
	\begin{equation}
		\supp\rho_{\infty,\epsilon} = \bar{\cS}_\epsilon\,,
	\end{equation}
	where $\cS_\epsilon$ is the bounded open set
	\begin{equation}
		\cS_\epsilon = \{\bx\in\mathbb{R}^d:\dist(\bx,\supp\rho_\infty)<\epsilon\}\,.
	\end{equation}
	Notice that
	\begin{equation}
		V[\rho_{\infty,\epsilon}] = W*\rho_\infty * \phi_\epsilon + U\,,
	\end{equation}
	in which $W*\rho_\infty$ and $U$ are continuous. Then, since we have the Euler-Lagrange condition $V[\rho_\infty]=W*\rho_\infty + U=C_0$ on $\supp\rho_\infty$, it is easy to see that
	\begin{equation}
		\lim_{\epsilon\rightarrow0}\sup_{\supp\rho_{\infty,\epsilon}}V[\rho_{\infty,\epsilon}] = C_0\,.
	\end{equation}
	Therefore, there exists $\epsilon_0>0$ such that for any $0<\epsilon<\epsilon_0$, $\rho_{\infty,\epsilon}$ satisfies the condition  \eqref{thm_main_1} for $\rho_\sharp$.
	
	Then we find $\epsilon\in (0,\epsilon_0)$ so that $|\partial \cS_\epsilon|=0$. To do this, notice that the function $t\mapsto |\cS_t|$ is increasing on $(0,\epsilon_0)$. By a rational number argument, there are at most countably many $t$ values such that $\lim_{\tau\rightarrow t^+} |\cS_\tau| > |\cS_t|$. Therefore, we may pick $\epsilon\in (0,\epsilon_0)$ so that $\lim_{\tau\rightarrow \epsilon^+} |\cS_\tau| = |\cS_\epsilon|$, which implies $|\partial \cS_\epsilon|=0$ since $\partial \cS_\epsilon = \bigcap_{\tau>\epsilon}(\cS_\tau\backslash \cS_\epsilon)$. For this $\epsilon$, if $D=\mathbb{R}^d$, then it is clear that $\rho_{\infty,\epsilon}$ satisfies all the assumptions for $\rho_\sharp$ in Theorem \ref{thm_main}. 
	
	If $D$ is a curved half-space as described in Theorem \ref{thm_main}, then we cannot take $\rho_\sharp=\rho_{\infty,\epsilon}$ because $\supp\rho_{\infty,\epsilon}$ may not be inside $D$. To handle this difficulty, we first notice that such $\epsilon$ can be chosen as small as we want. Then, we imitate the proof of Lemma \ref{lem_sufhalf} and take $\rho_\sharp(\hat{\bx},x_d)=\rho_{\infty,\epsilon}(\hat{\bx},x_d-(\epsilon+\omega(\epsilon)))$, so that $\supp\rho_\sharp$ lies in $D$. The regularity conditions are satisfied because $\rho_\sharp$ is a translation of $\rho_{\infty,\epsilon}$. Furthermore,
	\begin{equation}
		\Big|V[\rho_\sharp](\hat{\bx},x_d) - V[\rho_{\infty,\epsilon}]\big(\hat{\bx},x_d-(\epsilon+\omega(\epsilon))\big)\Big| \le \omega\big(\epsilon+\omega(\epsilon)\big)\,,
	\end{equation}
	which implies
	\begin{equation}
		\sup_{\supp\rho_\sharp}V[\rho_\sharp] \le \sup_{\supp\rho_{\infty,\epsilon}}V[\rho_{\infty,\epsilon}] + \omega\big(\epsilon+\omega(\epsilon)\big)\,.
	\end{equation}
	Therefore, by taking $\epsilon$ sufficiently small, the assumption \eqref{thm_main_1} is also satisfied.

\end{proof}

To prove Theorem \ref{thm_counter}, we first give a sufficient condition for non-existence of minimizers.

\begin{lemma}\label{lem_nonex}
	Assume $W(\bx) = -\frac{|\bx|^b}{b}$ with $-d<b<0$. Assume {\bf (U0)} and  $\lim_{|\bx|\rightarrow\infty}U(\bx)=0$. Suppose there exists a compactly supported $\rho\in\cM(\mathbb{R}^d)$ such that
	\begin{equation}\label{lem_nonex_1}
		V[\rho] = 0,\quad \rho\textnormal{-a.e.}\,;\qquad V[\rho] \ge 0,\quad\textnormal{a.e. }\mathbb{R}^d\,.
	\end{equation}
	Then $E_{W,\alpha U}$ does not admit a minimizer for any $0<\alpha<1$.
\end{lemma}

\begin{proof}
	
Denote
\begin{equation}
	V_\alpha[\cdot] = W*(\cdot) + \alpha U\,,
\end{equation}
as the generated potential for $(W,\alpha U)$. Then we have
\begin{equation}
	V_\alpha[\alpha \rho] = W*(\alpha \rho) + \alpha U = \alpha V[\rho]\,,
\end{equation}
and thus
\begin{equation}\label{ValphaEL}
	V_\alpha[\alpha \rho] = 0,\quad \rho\textnormal{-a.e.}\,;\qquad V_\alpha[\alpha \rho] \ge 0,\quad\textnormal{a.e. }\mathbb{R}^d\,.
\end{equation}

Suppose $\rho_*\in\cM(\mathbb{R}^d)$ is a minimizer of $E_{W,\alpha U}$. Then we have 
\begin{equation}\label{Erhost}
	E_{W,\alpha U}[\rho_*] \le E_{W,\alpha U}[\alpha \rho]\,,
\end{equation} 
(notice here that $\alpha\rho$ has total mass $\alpha<1$) because otherwise the $E_{W,\alpha U}$ energy of $\alpha \rho + \frac{1-\alpha}{|B(0;R)|}\chi_{B(0;R)} \in \cM(\mathbb{R}^d)$ with sufficiently large $R$ would be less than $E_{W,\alpha U}[\rho_*]$, a contradiction.

For $R>0$, define $\rho_{*,\epsilon,R}=(\rho_*\cdot\chi_{B(0;R)}) *\phi_\epsilon$ where $\phi_\epsilon$ is a mollifier. Define the linear interpolation
\begin{equation}
	\rho_t = (1-t)\alpha \rho + t \rho_{*,\epsilon,R},\quad 0\le t \le 1\,.
\end{equation}
\cite[Theorem 3.4]{Shu_convex} shows that $W$ is Fourier representable at level 0. Then, following the proof of \cite[Theorem 3.10]{Shu_convex}, we get $E_{W,\alpha U}[\alpha \rho+\rho_{*,\epsilon,R}] < \infty$, which justifies
\begin{equation}
	\frac{\rd^2}{\rd t^2} E_{W,\alpha U}[\rho_t] = 2E_W[\alpha\rho-\rho_{*,\epsilon,R}] = \int_{\mathbb{R}^d} \hat{W}(\xi)|\alpha\hat{\rho}(\xi)-\hat{\rho}_{*,\epsilon,R}(\xi)|^2\rd{\xi} \,.
\end{equation}
Also notice that
\begin{equation}
	\frac{\rd}{\rd t}\Big|_{t=0} E_{W,\alpha U}[\rho_t] = \int_{\mathbb{R}^d} V_\alpha[\alpha \rho]\rho_{*,\epsilon,R}\rd{\bx} \ge 0\,,
\end{equation}
by \eqref{ValphaEL} and the continuity of $\rho_{*,\epsilon,R}$. Therefore
\begin{equation}
	E_{W,\alpha U}[\rho_{*,\epsilon,R}] - E_{W,\alpha U}[\alpha \rho] \ge \frac{1}{2}\int_{\mathbb{R}^d} \hat{W}(\xi)|\alpha\hat{\rho}(\xi)-\hat{\rho}_{*,\epsilon,R}(\xi)|^2\rd{\xi}\,.
\end{equation}
Sending $\epsilon\rightarrow 0$ and $R\rightarrow \infty$, $\alpha\hat{\rho}-\hat{\rho}_{*,\epsilon,R}$ converges to $\alpha\hat{\rho}-\hat{\rho}_*$ uniformly on compact sets, where the latter is continuous and not identically zero. Combining with $\hat{W}(\xi) = c|\xi|^{-d-b},\,c>0$, we see that there exists $c_*>0$ such that $E_{W,\alpha U}[\rho_{*,\epsilon,R}]- E_{W,\alpha U}[\alpha \rho]\ge c_*$ for any $\epsilon$ sufficiently small and $R$ sufficiently large. Then taking limit (using \cite[Lemma 2.5]{CS21} to treat the mollification) gives $E_{W,\alpha U}[\rho_{*}]- E_{W,\alpha U}[\alpha \rho]\ge c_*$, contradicting \eqref{Erhost}.
	
\end{proof}

This proof indeed works for any $W$ that is {\bf (W0)}{\bf (W1)}, Fourier representable at level 0, and $\hat{W}$ is continuous on $\mathbb{R}^d\backslash \{0\}$ and strictly positive.

Then, Theorem \ref{thm_counter} is a direct consequence of the following corollary.
\begin{corollary}\label{cor_counter}
	Assume $d\ge 3$ and $W(\bx) = -\frac{|\bx|^b}{b}$ with $2-d<b<0$. Take a fixed mollifier $\phi$. Let
	\begin{equation}
		U = -W*\phi\,.
	\end{equation}
	Then, for $\alpha<1$ sufficiently close to 1, there exists a smooth and radial $\rho_\sharp\in\cM(\mathbb{R}^d)$ supported on a ball $\overline{B(0;R)}$ with $\sup_{B(0;R)} (W*\rho_\sharp+\alpha U) < 0$, but $E_{W,\alpha U}$ does not admit a minimizer.
\end{corollary}

\begin{proof}
By Lemma \ref{lem_nonex} with $\rho=\phi\in\cM(\mathbb{R}^d)$, it is clear that $V[\rho]=0$ on $\mathbb{R}^d$, satisfying \eqref{lem_nonex_1}, and thus $E_{W,\alpha U}$ does not admit a minimizer for any $0<\alpha<1$.

Then we denote
\begin{equation}
	\phi_t(\bx) = \frac{1}{t^d}\phi\big(\frac{\bx}{t}\big) \in \cM(\mathbb{R}^d)\,,
\end{equation}
for $t>0$, and calculate
\begin{equation}
	(W*\phi_t)(\bx) = \int_{\mathbb{R}^d} W(\bx-\by)\frac{1}{t^d}\phi\big(\frac{\by}{t}\big)\rd{\by} = \int_{\mathbb{R}^d} W(\bx-t\by)\phi(\by)\rd{\by}\,.
\end{equation}
Taking $t$-derivative and noticing that $\phi(\bx)=\varphi(|\bx|)$ is a radial function, we get
\begin{equation}\begin{split}
	\frac{\rd}{\rd{t}}(W*\phi_t)(\bx) = & -\int_{\mathbb{R}^d} \nabla W(\bx-t\by)\cdot\by\phi(\by)\rd{\by} \\
	= & -\int_0^1 \int_{S^{d-1}} \nabla W(\bx-t r \bz)\cdot r \bz \rd{S(\bz)} \,\varphi(r)r^{d-1}\rd{r} \\
	= & -\int_0^1 \int_{B(0;1)} \nabla_{\bz}\cdot\big(\nabla W(\bx-t r \bz)\big)\rd{\bz} \,\varphi(r)r^{d}\rd{r} \\
	= & \int_0^1 \int_{B(0;1)} \Delta W(\bx-t r \bz)\rd{\bz}\, \varphi(r)t r^{d+1}\rd{r}\,, \\
\end{split}\end{equation}
by using the spherical coordinates $\by=r \bz$, and noticing that the divergence theorem in the third equality is justified since the $\Delta W(\bx-t r \bz)$ is locally integrable in $\bz$. Since $\Delta W<0$, we see that the last quantity is negative, and there exists $c>0$ such that
\begin{equation}
	\frac{\rd}{\rd{t}}(W*\phi_t)(\bx)\le -c,\quad \forall 1\le t \le 2,\,|\bx| \le 2\,.
\end{equation} 
Therefore, taking $\rho_\sharp=\phi_2$, we see that $\rho_\sharp$ satisfies the regularity assumptions as desired, and (in view of $\phi_1=\phi$)
\begin{equation}
	W*\rho_\sharp+\alpha U \le W*\phi+\alpha U - c = (1-\alpha)(W*\phi) - c \le -\frac{c}{2} < 0,\quad \text{on }\overline{B(0;2)}=\supp\rho_\sharp\,,
\end{equation}
provided that
\begin{equation}
	\alpha \ge 1 - \frac{c}{2\sup(W*\phi)}\,.
\end{equation}
Then we see that $\sup_{B(0;2)} (W*\rho_\sharp+\alpha U)\le -\frac{c}{2}  < 0$.
\end{proof}

\section{Proof of complimentary results}\label{sec_compli}

\subsection{Proof of Theorem \ref{thm_exist1}}\label{sec_exist1}

We first give a simple lemma on a truncation and rescaling argument.
\begin{lemma}\label{lem_trunres}
	Assume {\bf (W0)}, {\bf (U0)}. Assume $\rho\in\cM(\mathbb{R}^d)$ with $E_{W,U}[\rho]<\infty$. Then
	\begin{equation}
		\lim_{R\rightarrow\infty} E_{W,U}\Big[\frac{1}{\rho(B(0;R))}\rho \chi_{B(0;R)}\Big] = E_{W,U}[\rho]\,.
	\end{equation}
\end{lemma}

\begin{proof}
By adding constants to $U$ and $W$, we may assume $\inf W = \inf U = 0$ without loss of generality.

First, since $\lim_{R\rightarrow\infty}\rho(B(0;R))=\rho(\mathbb{R}^d)=1$, we see that $\rho(B(0;R))>0$ for sufficiently large $R$. Then, using $\rho(B(0;R)) \le 1$, we get
\begin{equation}
	E_{W,U}\Big[\frac{1}{\rho(B(0;R))}\rho \chi_{B(0;R)}\Big] \le \frac{1}{\rho(B(0;R))^2}E_{W,U}[\rho\chi_{B(0;R)}]\le \frac{1}{\rho(B(0;R))^2}E_{W,U}[\rho]\,.
\end{equation}
We also have
\begin{equation}
	\liminf_{R\rightarrow\infty}E_{W,U}\Big[\frac{1}{\rho(B(0;R))}\rho \chi_{B(0;R)}\Big] \ge E_{W,U}[\rho]\,,
\end{equation}
by the lower-semicontinuity of $E_{W,U}$ with respect to weak convergence of measures. The conclusion follows.
\end{proof}

\begin{proof}[Proof of Theorem \ref{thm_exist1}]
	
By adding constants to $U$ and $W$ (which leaves \eqref{thm_exist1_1} invariant), we may assume $\inf W = \inf U = 0$ without loss of generality. 

We start by noticing that Lemma \ref{lem_trunres} and the assumption \eqref{thm_exist1_1} implies 
\begin{equation}\label{rhoReps}
	\inf_{\rho\in\cM(\overline{B(0;R)})}E_{W,U}[\rho] \le \epsilon+\inf_{\rho\in\cM(\mathbb{R}^d)}E_{W,U}[\rho],\quad \epsilon:=\frac{1}{4}\Big(U_\infty-2\inf_{\rho\in\cM(\mathbb{R}^d)}E_{W,U}[\rho]\Big)>0\,,
\end{equation}
if $R$ is sufficiently large (say, $R\ge R_0$).

We claim that there exists $R_*\ge R_0$ such that $\supp\rho_R \subset \overline{B(0;R_*)}$ for any $R\ge R_0$ and any minimizer $\rho_R$ of $E_{W,U}$ in $\cM(\overline{B(0;R)})$. To see this, we recall part of the Euler-Lagrange condition \eqref{lem_EL_2} for $\rho_R$, which takes the form
\begin{equation}\label{rhoREL}
	V[\rho_R] = C_0,\quad \rho_R\textnormal{-a.e. }\,,
\end{equation}
where $C_0$ depends on $R$ and $\rho_R$. Take any point $\bx_0\in \supp\rho_R$, then $\rho_R(B(\bx_0;1))>0$, and therefore there exists $\bx_1\in B(\bx_0;1)$ such that 
\begin{equation}\label{VrhoR}
	V[\rho_R](\bx_1) = (W*\rho_R)(\bx_1) + U(\bx_1) = C_0\,.
\end{equation}
Then we notice that
\begin{equation}\begin{split}
		E_{W,U}[\rho_R] = & \int_{\overline{B(0;R)}} \Big(\frac{1}{2}(W*\rho_R)(\bx) + U(\bx)\Big) \rd{\rho_R(\bx)} \\
		\ge & \int_{\overline{B(0;R)}} \Big(\frac{1}{2}(W*\rho_R)(\bx) + \frac{1}{2}U(\bx)\Big) \rd{\rho_R(\bx)} \\
		= & \frac{1}{2}\int_{\overline{B(0;R)}} V[\rho_R](\bx) \rd{\rho_R(\bx)} = \frac{1}{2}C_0\,,\\
\end{split}\end{equation}
where we used $U\ge 0$ and \eqref{rhoREL}. Therefore, using $W\ge 0$, \eqref{VrhoR} and \eqref{rhoReps}, we get
\begin{equation}
	U(\bx_1) \le C_0 \le 2E_{W,U}[\rho_R] = 2\inf_{\rho\in\cM(\overline{B(0;R)})}E_{W,U}[\rho] \le 2\inf_{\rho\in\cM(\mathbb{R}^d)}E_{W,U}[\rho] + 2\epsilon\,.
\end{equation}
Since 
\begin{equation}
	\lim_{|\bx|\rightarrow\infty}U(\bx) = U_\infty = 2\inf_{\rho\in\cM(\mathbb{R}^d)}E_{W,U}[\rho] + 4\epsilon\,,
\end{equation}
there exists $R_*>R_0$ such that $U(\bx) > 2\inf_{\rho\in\cM(\mathbb{R}^d)}E_{W,U}[\rho] + 3\epsilon$ for any $\bx$ with $|\bx|\ge R_*-1$. Then we see that $|\bx_1 | <R_*-1$, which implies $|\bx_0|<R_*$. This proves the claim.

Notice that this claim is also true if one replaces $\rho_R$ with a minimizer of $E_{W,U}$ in $\cM(\mathbb{R}^d)$. In other words, any minimizer of $E_{W,U}$ in $\cM(\mathbb{R}^d)$, if there exists any, is supported in $\overline{B(0;R_*)}$.

The existence of minimizer of $E_{W,U}$ in $\cM(\overline{B(0;R)})$ can be easily established using the compactness of $\overline{B(0;R)}$ and the lower-semicontinuity of $E_{W,U}$ with respect to weak convergence, as in \cite[Lemma 2.2]{CCP15}. Then we will prove that $\rho_{R_*}$, a minimizer of $E_{W,U}$ in $\cM(\overline{B(0;R_*)})$, is a minimizer of $E_{W,U}$ in $\cM(\mathbb{R}^d)$. In fact, we have proved that any minimizer of $E_{W,U}$ in $\cM(\overline{B(0;R)})$ with $R\ge R_*$ is supported in $\overline{B(0;R_*)}$. This implies 
\begin{equation}
	\inf_{\cM(\overline{B(0;R)})} E_{W,U} = \inf_{\cM(\overline{B(0;R_*)})} E_{W,U}\,,
\end{equation}
for any $R\ge R_*$. Then notice that
\begin{equation}
	\inf_{\cM(\mathbb{R}^d)} E_{W,U} = \lim_{R\rightarrow\infty}\inf_{\cM(\overline{B(0;R)})} E_{W,U}\,,
\end{equation}
due to Lemma \ref{lem_trunres}. It follows that $E_{W,U}[\rho_{R_*}]=\inf_{\cM(\overline{B(0;R_*)})} E_{W,U} = \inf_{\cM(\mathbb{R}^d)} E_{W,U}$, and thus $\rho_{R_*}$ is a minimizer of $E_{W,U}$ in $\cM(\overline{B(0;R_*)})$.

\end{proof}

\subsection{Proof of the new part of Theorem \ref{thm_exist2}}\label{sec_exist2}

For the existence of minimizers of $E_W$, the necessity and sufficiency of the condition \eqref{thm_exist2_0} is obtained in \cite[Theorems 3.1 and 3.2]{SST15} via the concentration-compactness lemma. If \eqref{thm_exist2_1} is assumed, the compact support property of any minimizer is obtained in \cite[Corollary 1.5]{CCP15}. Under an extra assumption:
\begin{equation}\begin{split}\label{assu_CCP}
	& \text{There exists $R>0$ such that $W$ is an increasing function on $\mathbb{R}^{k-1}\times [R,\infty)\times \mathbb{R}^{d-k}$} \\
	& \text{as a function of its $k$-th variable, for each $k\in \{1,\dots,d\}$,}
\end{split}\end{equation}
\cite[Theorem 1.4]{CCP15} gives an independent proof of the existence of minimizer, and derives an explicit upper bound for the diameter of the support of any minimizer in terms of $W$.

We will prove the same result as \cite[Theorem 1.4]{CCP15}, without the assumption \eqref{assu_CCP}, and thus Theorem \ref{thm_exist2} is obtained. We follow the approach of \cite{CCP15}, presenting a self-consistent proof of the existence of minimizer and estimate the diameter of its support. We use a new estimate (c.f. STEP 2 in the proof below) to guarantee that there cannot be a large gap between two clusters in a minimizer, as an improvement of \cite[Lemma 2.7]{CCP15} which relies on \eqref{assu_CCP}.

By adding constant, we may assume $W_\infty=0$ without loss of generality. We start by noticing that due to Lemma \ref{lem_trunres},  $\inf_{\rho\in\cM(\mathbb{R}^d)}E_W[\rho]<\frac{1}{2}W_\infty=0$ implies 
\begin{equation}\label{Erhoc0}
	\inf_{\rho\in\cM(\overline{B(0;R)})}E_W[\rho] \le -c_0,\quad c_0= -\frac{1}{2}\inf_{\rho\in\cM(\mathbb{R}^d)}E_W[\rho]>0\,,
\end{equation}
if $R$ is sufficiently large (say, $R\ge R_0$). 

We first recall a technical lemma.
\begin{lemma}[{\cite[Lemma 2.6]{CCP15}}]\label{lem_rhoRm}
	Under the same assumptions of Theorem \ref{thm_exist2} (with \eqref{thm_exist2_1}), there exists some large $R_1>0$ and small $m>0$ such that: for any minimizer $\rho_R$ of $E_W$ in $\cM(\overline{B(0;R)})$ for $R\ge R_0$, if $\bx_0\in\supp\rho$ then
	\begin{equation}\label{lem_rhoRm_1}
		\rho_R(B(\bx_0;R_1)) \ge m\,.
	\end{equation}
\end{lemma}

\begin{proof}[Proof of Theorem \ref{thm_exist2}]
	
We first prove that there exists $R_2$ such that: for any minimizer $\rho_R$ of $E_W$ in $\cM(\overline{B(0;R)})$ for $R\ge R_0$, there holds
\begin{equation}\label{xyR2}
	|\bx-\by|< R_2,\quad \forall \bx,\by\in\supp\rho_R\,.
\end{equation}
This will be done in STEPs 1-3 below.

{\bf STEP 1}: Construct balls which cover $\supp\rho$.

We start from a point $\bx_1\in \supp\rho_R$, and apply Lemma \ref{lem_rhoRm} to get $\rho_R(B(\bx_1;R_1)) \ge m$. Either $\supp\rho_R\subset B(\bx_1;2R_1)$, or we may take $\bx_2\in \supp\rho_R \backslash B(\bx_1;2R_1)$ and apply Lemma \ref{lem_rhoRm} to get $\rho_R(B(\bx_2;R_1)) \ge m$. 

We iterate this procedure as follows: if $\bx_1,\dots,\bx_{i-1}\in \supp\rho_R$ are already chosen, then either $\supp\rho_R\subset \bigcup_{j=1}^{k-1}B(\bx_j;2R_1)$, or we may take $\bx_i\in \supp\rho_R \backslash \bigcup_{j=1}^{i-1}B(\bx_j;2R_1)$ and apply Lemma \ref{lem_rhoRm} to get $\rho_R(B(\bx_i;R_1)) \ge m$.

Notice that this procedure guarantees $|\bx_i-\bx_j|\ge 2R_1$ if $i\ne j$. This implies that the balls $B(\bx_i;R_1)$ are disjoint. Since $\rho_R(B(\bx_i;R_1)) \ge m$ for every $i$ and $\rho_R$ is a probability measure, we see that this procedure has to terminate at some $n\le \frac{1}{m}$ with $\supp\rho_R\subset \bigcup_{j=1}^{n}B(\bx_j;2R_1)$.

{\bf STEP 2}: Estimate distances between balls.

Then we claim that there exists $R_3$ (large, to be determined) such that it is impossible to write 
\begin{equation}\label{R3S12}
	\bigcup_{j=1}^{n}B(\bx_j;2R_1) = S_1 \cup S_2,\quad S_i= \bigcup_{j\in J_i}B(\bx_j;2R_1) ,\,i=1,2\,,
\end{equation}
for some nonempty disjoint index sets $J_1,J_2$ such that $\dist(S_1,S_2)>R_3$. 

Suppose the contrary, then we write
\begin{equation}
	\rho_R = \rho_{R,1} + \rho_{R,2} = \rho\chi_{S_1} + \rho\chi_{S_2}\,,
\end{equation}
where both $\rho_{R,1}$ and $\rho_{R,2}$ are positive measures, whose masses $\rho_R(S_1)$ and $\rho_R(S_2)$ are between $m$ and $1-m$ with $\rho_R(S_1)+\rho_R(S_2)=1$. Then 
\begin{equation}\label{ErhoR12}
	E_W[\rho_R] = E_W[\rho_{R,1}] + E_W[\rho_{R,2}] + \int_{S_1}\int_{S_2}W(\bx-\by)\rd{\rho_{R,2}(\by)}\rd{\rho_{R,1}(\bx)}\,.
\end{equation}
By \eqref{Erhoc0}, we have $E_W[\rho_R]\le -c_0$ since $\rho_R$ is a minimizer of $E$ in $\cM(\overline{B(0;R)})$ with $R\ge R_0$. Since $\lim_{|\bx|\rightarrow\infty}W(\bx)=0$, we may take $R_3$ sufficiently large so that 
\begin{equation}\label{ErhoR12W}
	W(\bx)\ge -\epsilon c_0 ,\quad \forall |\bx|\ge R_3\,,
\end{equation}
where $\epsilon>0$ is a small parameter to be determined. Then
\begin{equation}\begin{split}
	E_W[\rho_{R,1}] + E_W[\rho_{R,2}] = & E_W[\rho_R] - \int_{S_1}\int_{S_2}W(\bx-\by)\rd{\rho_{R,2}(\by)}\rd{\rho_{R,1}(\bx)} \\
	\le & -c_0 - (-\epsilon c_0) = -(1-\epsilon)c_0 < 0\,,
\end{split}\end{equation}
since $\dist(S_1,S_2)>R_3$ is assumed. Then we construct two probability measures supported on $\overline{B(0;R)}$:
\begin{equation}
	\tilde{\rho}_{R,i} = \frac{1}{\rho_R(S_i)}\rho_{R,i},\quad i=1,2\,.
\end{equation} 
Then 
\begin{equation}
	E_W[\tilde{\rho}_{R,i}] = \frac{1}{\rho_R(S_i)^2}E_W[\rho_{R,i}]\,.
\end{equation}
We claim that 
\begin{equation}\label{rhoR12}
	\min\{E_W[\tilde{\rho}_{R,1}], E_W[\tilde{\rho}_{R,2}]\} \le \frac{1}{m^2+(1-m)^2} (E_W[\rho_{R,1}] + E_W[\rho_{R,2}])\,,
\end{equation}
where one notices that $\frac{1}{m^2+(1-m)^2}>1$. To see this, we first notice that if one of $E_W[\rho_{R,1}],E_W[\rho_{R,2}]$ is positive, say, $E_W[\rho_{R,1}]>0$, then $E_W[\rho_{R,2}]<0$ and it follows that
\begin{equation}
	E_W[\tilde{\rho}_{R,2}] =  \frac{1}{\rho_R(S_2)^2}E_W[\rho_{R,2}] \le \frac{1}{(1-m)^2}E_W[\rho_{R,2}]\le \frac{1}{(1-m)^2}(E_W[\rho_{R,1}] + E_W[\rho_{R,2}])\,,
\end{equation}
which gives \eqref{rhoR12}. If both $E_W[\rho_{R,1}],E_W[\rho_{R,2}]$ are nonpositive, then we write
\begin{equation}
	\min\{E_W[\tilde{\rho}_{R,1}], E_W[\tilde{\rho}_{R,2}]\} = (E_W[\rho_{R,1}] + E_W[\rho_{R,2}])\cdot \max\Big\{\frac{1}{\rho_R(S_1)^2}t,\frac{1}{\rho_R(S_2)^2}(1-t)\Big\}\,,
\end{equation}
where $t=\frac{E_W[\rho_{R,1}]}{E_W[\rho_{R,1}] + E_W[\rho_{R,2}]}\in [0,1]$. In the last maximum, the first quantity is increasing in $t$ and the second is decreasing. Therefore this maximum is minimized when the two quantities are equal, which gives
\begin{equation}\begin{split}
	\max\Big\{\frac{1}{\rho_R(S_1)^2}t,\frac{1}{\rho_R(S_2)^2}(1-t)\Big\} \ge & \frac{1}{\rho_R(S_1)^2}\cdot \frac{\frac{1}{\rho_R(S_2)^2}}{\frac{1}{\rho_R(S_1)^2}+\frac{1}{\rho_R(S_2)^2}} \\
	 = & \frac{1}{\rho_R(S_1)^2+\rho_R(S_2)^2} \\
	 \ge & \frac{1}{m^2+(1-m)^2}\,,
\end{split}\end{equation}
which implies \eqref{rhoR12}.

Combining \eqref{rhoR12} with \eqref{ErhoR12} and \eqref{ErhoR12W} we get
\begin{equation}\begin{split}
		\min\{E_W[\tilde{\rho}_{R,1}], E_W[\tilde{\rho}_{R,2}]\} \le & \frac{1}{m^2+(1-m)^2}(E_W[\rho_R] +\epsilon c_0)
		\\
		= & E_W[\rho_R] +\frac{1-m^2-(1-m)^2}{m^2+(1-m)^2}E_W[\rho_R] + \frac{1}{m^2+(1-m)^2} \epsilon c_0 \\
		\le & E_W[\rho_R] + \frac{1-m^2-(1-m)^2}{m^2+(1-m)^2}(-c_0) + \frac{1}{m^2+(1-m)^2} \epsilon c_0 \\
		= & E_W[\rho_R] + \frac{c_0}{m^2+(1-m)^2}\Big(-(1-m^2-(1-m)^2) +  \epsilon\Big)\,. \\
\end{split}\end{equation}
Taking
\begin{equation}
	\epsilon = \frac{1}{2}(1-m^2-(1-m)^2)>0\,,
\end{equation}
(and $R_3$ is determined by \eqref{ErhoR12W} accordingly) we see that $\min\{E_W[\tilde{\rho}_{R,1}], E_W[\tilde{\rho}_{R,2}]\} <  E_W[\rho_R]$, contradicting the fact that $\rho_R$ is a minimizer of $E_W$ in $\cM(\overline{B(0;R)})$. This proves the claim at the beginning of STEP 2.

{\bf STEP 3}: finalize. 

We start by noticing that some $\bx_j,\,j\in\{2,\dots,n\}$ satisfies $|\bx_1-\bx_j| \le R_3+4R_1$ because otherwise $|\bx_1-\bx_j| > R_3+4R_1$ for $j=2,\dots,n$, and one can choose $J_1=\{1\}$ in \eqref{R3S12} to make $\dist(S_1,S_2)>R_3$, which was proved to be impossible. Assume without loss of generality that this $j$ is 2. Then some $\bx_j,\,j\in\{3,\dots,n\}$ satisfies $|\bx_1-\bx_j| \le 2(R_3+4R_1)$ because otherwise $|\bx_1-\bx_j| > 2(R_3+4R_1)$ for $j=3,\dots,n$, and thus $|\bx_2-\bx_j| > R_3+4R_1$ for $j=3,\dots,n$. and choosing $J_1=\{1,2\}$ in \eqref{R3S12} leads to a contradiction. One can iterate this procedure $n-1$ times to show that $|\bx_1-\bx_j|\le (n-1)(R_3+4R_1)$ for every $j$. Therefore we obtain \eqref{xyR2} with $R_2 = 2(n-1)(R_3+4R_1)+4R_1$ since $\supp\rho_R$ is covered by $\{B(\bx_j;2R_1)\}$.

One can see from the proof of Lemma \ref{lem_rhoRm} (c.f. \cite{CCP15}) that it also applies to any minimizer of $E_W$ in $\cM(\mathbb{R}^d)$. Therefore STEPs 1-3 apply to any minimizer of $E_W$ in $\cM(\mathbb{R}^d)$ (if exists), showing that it is compactly supported at its support has diameter at most $R_2$.

{\bf STEP 4}: Existence of minimizer of $E_W$, and compact-support for any minimizer of $E_W$.

The existence of minimizer of $E_W$ in $\cM(\overline{B(0;R)})$ was established in \cite[Lemma 2.2]{CCP15}. We claim that a minimizer $\rho_{R_2}$ of $E_W$ in $\cM(\overline{B(0;R_2)})$ is a minimizer of $E_W$ in $\cM(\mathbb{R}^d)$. In fact, for any sufficiently large $R>R_2$ and a minimizer $\rho_R$ of $E_W$ in $\cM(\overline{B(0;R)})$, \eqref{xyR2} allows one to translate $\rho_R$ into $\overline{B(0;R_2)}$. Then, similar to the last paragraph of the proof of Theorem \ref{thm_exist1}, one can prove that $\rho_{R_2}$ is a minimizer of $E_W$ in $\cM(\mathbb{R}^d)$.

\end{proof}

{
\section*{Acknowledgement}
The author would like to thank Edward B. Saff for helpful suggestions on a draft of this paper.
}

\bibliographystyle{alpha}
\bibliography{minimizer_book_bib.bib}

\end{document}